\documentclass[a4paper]{article}

\usepackage{authblk}
\usepackage{amsmath,amssymb,amsfonts,amsthm}
\usepackage{bbm,booktabs}
\usepackage{color}
\usepackage{enumerate}
\usepackage{enumitem}
\usepackage{fullpage}
\usepackage[acronym]{glossaries}
\usepackage{graphicx}
\usepackage{multirow}
\usepackage{nicefrac}
\usepackage[percent]{overpic}
\usepackage{setspace}
\usepackage{tikz}
\usepackage{units}

\usepackage{epstopdf} 
\usepackage{hyperref}

\definecolor{MyDarkBlue}{rgb}{0,0.08,0.50}
\definecolor{BrickRed}{rgb}{0.65,0.08,0}

\hypersetup{
colorlinks=true,       		
    linkcolor=MyDarkBlue,   
    citecolor=BrickRed,     
    filecolor=red,      	
    urlcolor=cyan           
}

\newcommand{\pr}{\rightarrow}

\newcommand{\ba}{\begin{array}}
\newcommand{\ea}{\end{array}}

\newcommand{\vart}{\vartheta}
\newcommand{\varp}{\phi}
\newcommand{\eps}{\varepsilon}

\newenvironment{inspring}[1]%
{\begin{list}{}{\setlength{\rightmargin}{0cm}
                \setlength{\listparindent}{0cm}
                \settowidth{\labelwidth}{\mbox{#1}}
                \setlength{\leftmargin}{1.1\labelwidth}
                \setlength{\labelsep}{.1\labelwidth}}}%
{\end{list}}

\newcommand{\bi}[1]{\begin{inspring}{#1}}
\newcommand{\ei}{\end{inspring}}

\newcommand{\dprod}{\displaystyle \prod}

\newcommand{\beq}{\begin{equation}}
\newcommand{\eq}{\end{equation}}
\catcode`\@=11

\font\tenmsa=msam10 \font\sevenmsa=msam7 \font\fivemsa=msam5
\font\tenmsb=msbm10 \font\sevenmsb=msbm7 \font\fivemsb=msbm5
\newfam\msafam
\newfam\msbfam
\textfont\msafam=\tenmsa  \scriptfont\msafam=\sevenmsa
  \scriptscriptfont\msafam=\fivemsa
\textfont\msbfam=\tenmsb  \scriptfont\msbfam=\sevenmsb
  \scriptscriptfont\msbfam=\fivemsb

\def\Bbb{\ifmmode\let\next\Bbb@\else
 \def\next{\errmessage{Use \string\Bbb\space only in math mode}}\fi\next}
\def\Bbb@#1{{\Bbb@@{#1}}}
\def\Bbb@@#1{\fam\msbfam#1}
\newcommand{\dR}{{\Bbb R}}

\newcommand{\dC}{{\Bbb C}}

\newtheorem{thm}{Theorem}
\newtheorem{prop}[thm]{Proposition}
\newtheorem{lem}[thm]{Lemma}
\theoremstyle{definition}

\theoremstyle{remark}

\numberwithin{thm}{section}
\numberwithin{equation}{section}

\newcommand{\MM}{\chi}
\newcommand{\e}{{\rm e}}

\newcommand{\eqan}[1]{\begin{align} #1 \end{align}}

\usepackage{dsfont} 

\newcommand{\sA}{ {\ensuremath{x}} } 
\newcommand{\sB}{ {\ensuremath{y}} } 
\newcommand{\stateSpace}{ \Omega }

\newcommand{\vC}[2]{ #1_{#2} }
\newcommand{\seq}[2]{ #1^{(#2)} }
\newcommand{\indicator}[1]{ \mathds{1} [ #1 ] }
\newcommand{\process}[2]{ #1 }
\newcommand{\bigO}[1]{ O(#1) }

\newcommand{\naturalNumbersZero}{ \mathbb{N}_0 }
\newcommand{\realNumbers}{ \mathbb{R} }

\newcommand{\refFigure}[1]{{\textrm{Figure~\ref{#1}}}}
\newcommand{\refTable}[1]{{\textrm{Table~\ref{#1}}}}
\newcommand{\refEquation}[1]{{\textrm{\eqref{#1}}}}

\newcommand{\refTheorem}[1]{{\textrm{Theorem~\ref{#1}}}}

\newcommand{\refProposition}[1]{{\textrm{Proposition~\ref{#1}}}}
\newcommand{\refLemma}[1]{{\textrm{Lemma~\ref{#1}}}}

\newcommand{\refSection}[1]{{\textrm{\S\ref{#1}}}}
\newcommand{\refAppendixSection}[1]{{\textrm{\ref{#1}}}}

\begin{document}

\title{Scaled control in the QED regime}

\date{July 4, 2013}

\author{A.J.E.M.\ Janssen}

\author{J.S.H.\ van Leeuwaarden}

\author{Jaron Sanders\footnote{Electronic addresses: \texttt{a.j.e.m.janssen@tue.nl}, \texttt{j.s.h.v.leeuwaarden@tue.nl}, and \texttt{jaron.sanders@tue.nl}}}


\affil{Department of Mathematics \& Computer Science, Eindhoven University of Technology, P.O.\ Box 513, 5600 MB Eindhoven, The Netherlands}


\makeglossaries
\newacronym{QED}{QED}{Quality-and-Efficiency-Driven}
\newacronym{OU}{OU}{Ornstein-Uhlenbeck}
\newacronym{EM}{EM}{Euler-Maclaurin}

\maketitle

\begin{abstract}
We develop many-server asymptotics in the \gls{QED} regime for models with admission control. The admission control, designed to reduce the incoming traffic in periods of congestion, scales with the size of the system. For a class of Markovian models with this scaled control, we identify the \gls{QED} limits for two stationary performance measures. We also derive corrected \gls{QED} approximations, generalizing earlier results for the Erlang B, C and A models. These results are useful for the dimensioning of large systems equipped with an active control policy. In particular, the corrected approximations can be leveraged to establish the optimality gaps related to square-root staffing and asymptotic dimensioning with admission control.

\vspace{1em}


\end{abstract}

\section{Introduction} \label{sec1}

Many-server systems have the capability of combining large capacity with high utilization while maintaining satisfactory system performance. This potential for achieving economies of scale is perhaps most pronounced in the \gls{QED} regime, or Halfin-Whitt regime. Halfin and Whitt \cite{halfinwhitt} were the first to study the \gls{QED} regime for the $GI/M/s$ system. Assuming that customers require an exponential service time with mean $1$, the \gls{QED} regime refers to the situation that the arrival rate of customers $\lambda$ and the numbers of servers $s$ are increased in such a way that the traffic intensity $\rho = \lambda/s$ approaches one and
\begin{equation}\label{1.1}
(1-\rho)\sqrt{s}\rightarrow \gamma, \quad \gamma\in\mathbb{R}.
\end{equation}

The scaling \refEquation{1.1} is effective because the probability of delay converges to a non-degenerate limit away from both zero and one. Limit theorems for other, more general systems are obtained in \cite{garnett,jelenkovic,maglaraszeevi,manmom,reed}, and in all these cases, the limiting probability of delay remains in the interval $(0,1)$. In fact, not only the probability of delay, but many other performance characteristics or objective functions are shown to behave (near) optimally in the \gls{QED} regime, see for example \cite{borst,gans}. An important reason for this near optimal behavior are the relatively small fluctuations of the queue-length process.

This can be understood in the following way. Let $X_s(t)=(Q_s(t)-s)/\sqrt{s}$ denote a sequence of normalized processes, with $Q_s(t)$ the process describing the number of customers in the system over time. When $X_s(t)>0$, it is equal to the scaled total number of customers in the queue, whereas when $X_s(t)<0$, it is equal to the scaled number of idle servers. Halfin and Whitt showed for the $GI/M/s$ system how under \refEquation{1.1}, $X_s(t)$ converges to a diffusion process $X(t)$ on $\realNumbers$, that behaves like a Brownian motion with drift above zero and like an \gls{OU} process below zero, and that has a non-degenerate stationary distribution. This shows that the natural scale of $Q_s(t)-s$ is of the order $\sqrt{s}$. More precisely, the queue length is of the order $\sqrt{s}$, as well as the number of idle servers.

This paper adds to the \gls{QED} regime the feature of state-dependent control, by considering a control policy that lets an arriving customer enter the system according to some probability depending on the queue length. In particular, a customer meeting upon arrival $k$ other waiting customers is admitted with probability $p_s(k)$, and we allow for a wide range of such control policies characterized by $\{p_s(k)\}_{k\in\naturalNumbersZero}$ with $\naturalNumbersZero=\{0,1,\dots\}$. An important property of this control is that it is allowed to scale with the system size $s$. Consider for example {\it finite-buffer control}, in which new customers are rejected when the queue length equals $N$, so that $p_s(k)=1$ for $k<N$ and $p_s(k)=0$ for $k\geq N$. A finite-capacity effect in the \gls{QED} regime that is neither dominant nor negligible occurs when $N \approx \eta\sqrt{s}$ with $\eta > 0$, because the natural scale of the queue length is $\sqrt{s}$. A similar threshold in the context of many-server systems in the \gls{QED} regime has been considered in \cite{armonymaglaras,masseywallace,whittrigor,whittprox}.

We introduce a class of \gls{QED}-specific control policies $\{p_s(k)\}_{k\in\naturalNumbersZero}$ designed, like the finite-buffer control, to control the fluctuations of $Q_s(t)$ around $s$. To this end, we consider control policies for which $p_s( x \sqrt{s} ) \approx 1-a(x)/\sqrt{s}$ when $x > 0$. Here, $a$ denotes a non-negative and non-decreasing function. While almost all customers are admitted as $s \to \infty$, this control is specifically designed for having a decisive influence on the system performance in the \gls{QED} regime. We also provide an in-depth discussion of two canonical examples. The first is {\it modified-drift control} given by $a(x)= \vart >-\gamma$, which is shown to effectively change the \gls{QED} parameter $\gamma$ in \refEquation{1.1} into $\gamma + \vart$. The second example is {\it Erlang A control} given by $a(x)=\vart x$, for which the system behavior is shown to be intimately related with that of the Erlang A model in which waiting customers abandon the system after an exponential time with mean $1/\vart$.

Our class of \gls{QED}-specific control policies stretches much beyond these two examples. In principle, we can choose the control such that, under Markovian assumptions, the stochastic-process limit for the normalized queue-length process changes the Brownian motion in the upper half plane (corresponding to the system without control), into a diffusion process with drift $-\gamma -a(x)$ in state $x\geq 0$. We give a formal proof of this process-level result.

We next consider the controlled \gls{QED} system in the stationary regime and derive the \gls{QED} limits for the probability of delay and the probability of rejection. Typically, such results can be obtained by using the central limit theorem and case-specific arguments, see for example \cite{garnett,halfinwhitt,jelenkovic,masseywallace}. However, we take a different approach, aiming for new asymptotic expansions for the probability of delay and the probability of rejection. The first terms of these expansions are the \gls{QED} limits, and the higher-order terms are refinements to these \gls{QED} limits for finite $s$. This generalizes earlier results on the Erlang B, C and A models \cite{erlangb,erlangc,erlanga}.

Conceptually, we develop a unifying approach to derive such expansions for these control policies. A crucial step in our analysis is to rewrite the stationary distribution in terms of a Laplace transform that contains all specific information about the control policy. Mathematically, establishing the expansions requires an application of \gls{EM} summation, essentially identifying the error terms caused by replacing a series expression in the stationary distribution by the Laplace transform. In this paper we focus on the probability of delay and the probability of rejection, but it is fairly straightforward using the same approach to obtain similar results for other characteristics of the stationary distribution, such as the mean and the cumulative distribution function.

The paper is structured as follows. In \refSection{sec2} we introduce the many-server system with admission control and derive the stability condition under which the stationary distribution exists. In \refSection{sec3} we discuss in detail the \gls{QED} scaled control. We introduce a {\it global} control for managing the overall system fluctuations, and a {\it local} control that entails a precise form of $p_s(k)$. For both the global and local control, we derive the stability condition and the stochastic-process limit for the normalized queue-length process in terms of a diffusion process. In \refSection{sec4} we derive \gls{QED} approximations for systems with global control. Hereto, we
enroll our concept of describing the stationary distribution in terms of a Laplace transform and using \gls{EM} summation to derive the expansions. In \refSection{sec5} we derive \gls{QED} approximations for local control, making heavy use of the intimate connection with global control and the tools developed in \refSection{sec4}. For demonstrational purposes we also provide some numerical results for the Erlang A control. In \refSection{sec6} we discuss the potential applications of the results obtained in this paper.

\section{Many-server systems with admission control} \label{sec2}

Consider a system with $s$ parallel servers to which customers arrive according to a Poisson process with rate $\lambda$. The service times of customers  are assumed exponentially distributed with mean $1$. A control policy dictates whether or not a customer is admitted to system. A customer that finds upon arrival $k$ other waiting customers in the system is allowed to join the queue with probability $p_s(k)$ and is rejected with probability $1-p_s(k)$. In this way, the sequence $\{p_s(k)\}_{k\in\naturalNumbersZero}$ defines the control policy. Since we are interested in large, many-server systems, working at critical load and hence serving many customers, the probability $p_s(k)$ should be interpreted as the fraction of customers admitted in state $s + k$.

Under these Markovian assumptions, and assuming that all interarrival times and service times are mutually independent, this gives rise to a  birth--death process  $Q_s(t)$  describing the number of customers in the system over time. The birth rates are $\lambda$ for states $k=0,1,\ldots ,s$ and $\lambda\cdot p_s(k-s)$ for states $k=s,s+1, \ldots \,$. The death rate in state $k$ equals $\min{ \{ k,s \} }$ for states $k=1, 2, \ldots \,$. Assuming the stationary distribution to exist, with $\pi_k=\lim_{t\pr\infty}\mathbb{P}(Q_s(t)=k)$, it follows from solving the balance equations that
\begin{equation} \label{2.1}
\pi_k=\left\{\ba{lll}
\pi_0 \frac{(s\rho)^k}{k!}, & ~~~k=1,2, \ldots, s, \\
\pi_0 \frac{ s^s \rho^k }{ s! } \dprod_{i=0}^{k-s-1}\,p_s(i),  & ~~~k=s+1,s+2, \ldots .
\ea\right.
\end{equation}
Here
\begin{equation} \label{2.2}
\rho=\frac{\lambda}{s},\quad \pi_0^{-1}=\sum_{k=0}^s\,\frac{(s\rho)^k}{k!}+\frac{(s\rho)^s}{s!}F_s(\rho)
\end{equation}
with
\begin{equation} \label{2.3}
F_s(\rho)=\sum_{n=0}^{\infty}\,p_s(0)\cdot \ldots \cdot p_s(n)\,\rho^{n+1}.
\end{equation}

\subsection{Stability} \label{subsec2.1}

The wide class of allowed control policies renders it necessary to carefully investigate the precise conditions under which the controlled system is stable. From \refEquation{2.1}--\refEquation{2.3}, we conclude that the stationary distribution exists if and only if $\{p_s(k)\}_{k\in\naturalNumbersZero}$ and $\rho$ are such that $F_s(\rho)<\infty$. Let
\begin{equation} \label{2.4}
P_s := \underset{ n \pr \infty }{ \limsup } \, \left( p_s(0)\cdot \ldots \cdot p_s(n) \right)^{\frac{1}{n+1}},
\end{equation}
and set $1/P_s=\infty$ when $P_s=0$. We then see that $F_s(\rho) < \infty$ when
\begin{equation} \label{2.5}
0\leq\rho<\frac{1}{P_s}.
\end{equation}
For convenience, we henceforth assume that
\begin{equation} \label{2.6}
\lim_{\rho\uparrow 1/P_s}\,F_s(\rho)=\infty,
\end{equation}
so that the stationary distribution exists if and only if \refEquation{2.5} holds. The case  $\lim_{\rho\uparrow 1/P_s}\,F_s(\rho)<\infty$ (as considered for example in \cite{ref1,ref2}) can also be considered in the present context, but leads to some complications that distract attention from the bottom line of the exposition.

\subsection{Performance measures} \label{subsec2.2}

We consider in this paper two performance measures, viz.\ the stationary probability $D_s(\rho)$ that an arriving customer finds all servers occupied, and the stationary probability $D_s^R(\rho)$ that an arriving customer is rejected. In terms of $\pi_k$ and $p_s(k)$, these stationary probabilities are given by $D_s(\rho) = \sum_{k=s}^{\infty} \pi_k$ and $D_s^R(\rho) = \sum_{k=s}^{\infty} \pi_k(1-p_s(k))$.
Denoting the Erlang B formula by
\begin{equation} \label{2.21}
B_s(\rho)=\frac{(s\rho)^s/s!}{\sum_{k=0}^s\,(s\rho)^k/k!},
\end{equation}
we express $D_s(\rho)$ and $D_s^R(\rho)$ in terms of $B_s(\rho)$ and $F_s(\rho)$ as
\begin{equation} \label{2.22}
D_s(\rho)=\frac{1+F_s(\rho)}{B_s^{-1}(\rho)+F_s(\rho)}
\end{equation}
and
\begin{equation} \label{2.23}
D_s^R(\rho)=\frac{1+(1-\rho^{-1})\,F_s(\rho)}{B_s^{-1}(\rho)+F_s(\rho)}.
\end{equation}

There are two extreme control policies. The first is the control that denies all customers access whenever all servers are occupied, i.e.\ $p_s(k)=0$ for $k\in\naturalNumbersZero$. This is in fact the Erlang B system. Then, $F_s(\rho)=0$ and \refEquation{2.22}, \refEquation{2.23} indeed give $D_s(\rho) = D_s^R(\rho) = B_s(\rho)$. The other is the control that allows all customers access, i.e.\ $p_s(k)=1$ for $k\in\naturalNumbersZero$, known as the Erlang C system. 
Equation \refEquation{2.3} gives $F_s(\rho)=\rho/(1-\rho)$ for $0 \leq \rho < 1$.
Subsequently, \refEquation{2.23} gives $D_s^R(\rho)=0$ (a customer is never rejected) and expression \refEquation{2.22} reduces to the Erlang C formula
\begin{equation} \label{2.25}
C_s(\rho)=\frac{\frac{(s\rho)^s}{s!(1-\rho)}}{\sum_{k=0}^{s-1}\,\frac{(s\rho)^k}{k!}+\frac{(s\rho)^s}{s! \,(1-\rho)}}.
\end{equation}

\section{\glsentrytext{QED} scaled control}\label{sec3}

To enforce the \gls{QED} regime in \refEquation{1.1} we henceforth couple $\lambda$ and $s$ according to
\begin{equation} \label{2.8}
\rho=\frac{\lambda}{s}=1-\frac{\gamma}{\sqrt{s}}\Leftrightarrow\lambda=s-\gamma\,\sqrt{s},\quad  \gamma\in\dR.
\end{equation}
We next introduce two types of control, referred to as \emph{global} and \emph{local} control, both designed to reduce the incoming traffic in periods of congestion.

\subsection{Global control}
Recall \refEquation{2.3} and let
\begin{equation} \label{2.10}
q_s(n):=p_s(0)\cdot \ldots \cdot p_s(n), \quad n \in \naturalNumbersZero
\end{equation}
be the coefficient of $\rho^{n+1}$ in $F_s(\rho)$. For $n \in \naturalNumbersZero$, $q_s(n)$ is roughly equal to the probability that a (fictitious) batch arrival of $n$ customers is allowed as a whole to enter the system, given that all servers are busy and that the waiting queue is empty. Since in the \gls{QED} regime queue lengths are of the order $\sqrt{s}$, it is natural to consider control policies such that $q_s(n)$ scales with $s$ in a $\sqrt{s}$-manner as well. One way to achieve this is by choosing $q_s(n)$ of the form
\begin{equation} \label{2.11}
q_s(n)=f\Bigl( \frac{n+1}{\sqrt{s}} \Bigr), \quad n \in \naturalNumbersZero
\end{equation}
for $s\geq1$, where $f(x)$, henceforth referred to as {\it scaling profile}, is a non-negative, non-increasing function of $x\geq0$ with $f(0)=1$. With global control we mean that the admission control is defined through $q_s(n)$ in \eqref{2.11}.
A key example is what we have called modified-drift control, in which case
\begin{equation} \label{mdc}
p_s(k)=p^{\frac{1}{\sqrt{s}}}, \ p\in(0,\infty), \quad q_s(n)=p^{\frac{n+1}{\sqrt{s}}}=f\Bigl(\frac{n+1}{\sqrt{s}}\Bigr) \ {\rm with} \ f(x)=p^x.
\end{equation}
It appears that many practical admission policies fit into the  Ansatz \refEquation{2.11}, or do so in a limit sense as $s\pr\infty$. This is the case for the class of local control, as discussed next.

\subsection{Local control}

While the global control is defined via $q_s(n)$, we also introduce a local control, that for each state $k$, defines the probability of admitting a new customer as
\begin{equation} \label{1.2}
p_s(k)=\frac{1}{1+\frac{1}{\sqrt{s}}a(\frac{k+1}{\sqrt{s}})}, \quad k \in \naturalNumbersZero,
\end{equation}
with $a(x)$ a non-negative, non-decreasing function of $x\geq 0$. A special case is Erlang A control $a(x)=\vart x$, which gives $p_s(k)=1 / (1+(k+1) \vart / s )$. In this case the stationary distribution is identical to that of an $M/M/s+M$ system (or Erlang A model), with the feature that customers that are waiting in the queue abandon the system after exponentially distributed times with mean $1/\vart$. Garnett et al.~\cite{garnett} obtained the diffusion limit for the Erlang A model in the \gls{QED} regime, and the limiting diffusion process turned out to be a combination of two \gls{OU} processes with different restraining forces, depending on whether the process is below or above zero.

Note that setting $a(x)=0$ leads to the ordinary $M/M/s$ system considered in \cite{halfinwhitt} with in the \gls{QED} regime as limiting process a Brownian motion in the upper half plane. Depending on $a$, i.e.~the type of control, one gets a specific limiting behavior in the upper half plane, described by Brownian motion, an \gls{OU} process, or some other type of diffusion process with drift $- \gamma - a(x)$ in state $x \geq 0$. We give a formal proof of this process-level convergence in \refSection{sec:main}.

\subsection{Connection between local and global control}\label{conn}


There is a fundamental relation between local and global control. By substituting \refEquation{1.2} into \refEquation{2.10},  rewriting the product and using Taylor expansion, we see that
\begin{align} \label{6.2}
q_s(n) & =  p_s(0)\cdot \ldots \cdot p_s(n)=\exp\Bigl({-}\sum_{k=0}^n\,{\rm ln}\Bigl(1+\frac{1}{\sqrt{s}}\,a\Bigl( \frac{k+1}{\sqrt{s}}\Bigr)\Bigr)\Bigr) \nonumber \\
&=  \exp\Bigl(\frac{-1}{\sqrt{s}}\:\sum_{k=0}^n\,a\Bigl(\frac{k+1}{\sqrt{s}}\Bigr)+ O\Bigl(\frac1s\: \sum_{k=0}^n\,a^2\Bigl(\frac{k+1}{\sqrt{s}}\Bigr)\Bigr)\Bigr), \quad n \in \naturalNumbersZero.
\end{align}
For large $s$ and under mild conditions on $a$, the last expression in (\ref{6.2}) can be approximated by
\begin{equation} \label{6.3}
\exp\Bigl({-}\,\int_0^{\frac{n+1}{\sqrt{s}}}\,a(y)\,dy+O\Bigl(\frac{1}{\sqrt{s}}\:\int_0^{\frac{n+1}{\sqrt{s}}}\, a^2(y)\,dy\Bigr)\Bigr),
\end{equation}
which will be discussed in more detail in \refSection{logl}.
We get the approximation
\begin{equation} \label{6.4}
q_s(n) \approx f\Bigl(\frac{n+1}{\sqrt{s}}\Bigr), \quad n \in \naturalNumbersZero,
\end{equation}
where
\begin{equation} \label{6.5}
f(x)=\exp\Bigl({-}\,\int_0^x\,a(y)\,dy\Bigr),\quad x\geq0.
\end{equation}

The validity range and the approximation error in (\ref{6.4}) depend on the particular form of $a$, which will be discussed in detail in \refSection{sec5}. Also, \refEquation{6.5} implies that $f$ and $a$ are related as
\begin{equation} \label{2.12}
a(x)={-}\,\frac{f'(x)}{f(x)},\quad x\geq0.
\end{equation}
From here onwards we assume that $f$ in \eqref{2.11} and $a$ in \eqref{1.2} are indeed related according to \refEquation{6.5} and \refEquation{2.12}. We can then show that both local and global control have a similar impact on a system, characterized by
\begin{equation}
\vC{p}{s}(k) \approx 1 - \frac{1}{\sqrt{s}}a \Bigl( \frac{k+1}{\sqrt{s}} \Bigr), \quad k \in \naturalNumbersZero. 
\label{eqn:Asymptotic_behavior_of_each_individual_control_probability}
\end{equation}

In \refSection{subsec2.2} we discussed how our class of control policies can cover the entire range between the Erlang B model and the Erlang C model.
Let us demonstrate that for the modified-drift control described in \refEquation{mdc} that admits a customer when all servers are busy with probability $p^{1 / \sqrt{s}}$, where $p \in (0,1)$. \refFigure{fixedact} shows for fixed $\rho=0.99$ the delay probability $D_s(\rho)$ as a function of $p$. Here we show both global control $f(x) = p^x$ and the local control counterpart $a(x) = -\ln p$. Notice the relatively small difference between global and local control, which would be even smaller for larger values of $s$. 

\begin{figure}[!hbtp]
\begin{center}
\begin{tikzpicture}
\node[anchor=south west,inner sep=0] at (0,0) {\includegraphics[width=3.3in, keepaspectratio]{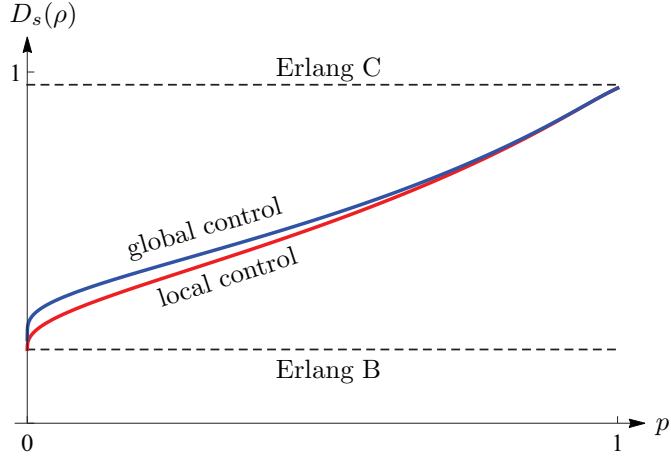}};
\node[anchor=south west,inner sep=0] at (8.5,0.3) {$p$};
\node[anchor=south west,inner sep=0] at (0,5.7) {$D_s(\rho)$};
\node[anchor=south west,inner sep=0] at (3.5,1.0) {Erlang B};
\node[anchor=south west,inner sep=0] at (3.5,5) {Erlang C};
\node[anchor=south west,inner sep=0,rotate=16] at (1.6,2.6) {global control};
\node[anchor=south west,inner sep=0,rotate=18] at (2,2) {local control};
\end{tikzpicture}
\caption{The stationary probability of delay for global control $f(x) = p^x$ and local control $a(x) = -\ln p$ for $s = 10$ and $\rho = 0.99$.}
\label{fixedact}
\end{center}
\end{figure}

\subsection{Stability with control}
Now that we have established the connection between global and local control via the relations \refEquation{6.5} and \refEquation{2.12}, we next show that the stability conditions for the systems with these respective controls are similar as well.

Define the Laplace transform of the scaling profile $f$ as
\begin{equation} \label{2.13}
{\cal L}(\gamma) = \int_0^{\infty} e^{-\gamma x} f(x) dx, \quad \gamma > \gamma_{\min},
\end{equation}
where $\gamma_{\min} = \inf\{\gamma\in\dR|{\cal L}(\gamma)<\infty\}$.
From \refEquation{6.5}, it follows that $\gamma_{\min} = - \lim_{x\to\infty} a(x)$,
and since $a$ is non-decreasing, we have
\begin{equation} \label{2.15}
\lim_{\gamma\downarrow\gamma_{\min}} {\cal L}(\gamma) = \infty.
\end{equation}

In \refAppendixSection{proofprop2.1}, we derive the following stability condition for global and local control in terms of $\gamma_{\min}$. For large $s$, the two stability conditions are almost identical.

\begin{prop}[Stability conditions] \label{prop2.1}
Assume \refEquation{6.5} and \refEquation{2.12}. The stationary distribution \refEquation{2.1} exists for
\begin{itemize}[noitemsep,nolistsep]
\item[{\rm (i)}] the global control \refEquation{2.11} if and only if $0 \leq \rho < e^{- \gamma_{\min} / \sqrt{s} } = 1 - \gamma_{\min} / \sqrt{s} + \bigO{ 1/s }$;
\item[{\rm(ii)}] the local control \refEquation{1.2} if and only if $0 \leq \rho < 1 - \gamma_{\min} / \sqrt{s}$.
\end{itemize}
\end{prop}

\subsection{Stochastic-process limit}\label{sec:main}

We now derive using the local control in \refEquation{1.2} a stochastic-process limit, which provides additional insight into the roles of the function $a$ and the Laplace transform $\mathcal L$.

Let $Q_s(t)$ denote  the process describing the number of customers present in the system over time. The subscript $s$ is attached to all relevant quantities to denote their dependence on the size of the system. We obtain a scaling limit for the sequence of normalized processes $X_s(t)=(Q_s(t)-s)/\sqrt{s}$. Let ``$\Rightarrow$'' denote weak convergence in the space $D[0,\infty)$ or convergence in distribution. The next result is proved in \refAppendixSection{proof:dp}.

\begin{prop}[Weak convergence to a diffusion process]
\label{prop:Weak_convergence_to_a_diffusion_process}
Assume \refEquation{1.1} and \refEquation{1.2}. If $a$ is continuous and bounded on every compact subinterval $I$ of $\realNumbers$, and $X_s(0)\Rightarrow X(0)\in\mathbb{R}$, then for every $t\geq 0$, as $s\to \infty$,
\begin{equation}
X_s(t)\Rightarrow X(t),
\end{equation}
where the limit $X(t)$ is the diffusion process with infinitesimal drift $m(x)$ given by
\begin{equation} \label{eqn:Drift_of_the_diffusion_process}
m(x)=
\begin{cases}
-\gamma-x, & x<0, \\
-\gamma-a(x), & x \geq 0, \\
\end{cases}
\end{equation}
and constant infinitesimal variance $\sigma^2(x)=2$.
\end{prop}

\refProposition{prop:Weak_convergence_to_a_diffusion_process} sheds light on the effect of the control $\vC{p}{s}(k)$ as $s$ becomes large. It shows that for local control \refEquation{1.2}, which is asymptotically of the form \refEquation{eqn:Asymptotic_behavior_of_each_individual_control_probability},  the process $\vC{Q}{s}(t)$ approximately behaves as $s + X(t) \sqrt{s}$, where $X(t)$ is a diffusion process with drift $- \gamma - a(x)$ for $x \geq 0$ and an \gls{OU} process with drift $- \gamma - x$ for $x < 0$.

The stationary distribution of $X(t)$ is easy to derive. Denote the probability density function of the standard normal distribution by $\phi(x)$, and its cumulative distribution function by $\Phi(x) = \int_{-\infty}^x \phi(u) du$.

\begin{prop}[Stationary distribution of the diffusion process]
\label{prop:Stationary_distribution_for_diffusion_process}
The density function $\omega(x)$ of the stationary distribution for $X(t)$ is given by
\begin{equation}\label{222}
\omega(x)
=
\begin{cases}
C(\gamma)\frac{\phi(x+\gamma)}{\phi(\gamma)}, & x < 0, \\
C(\gamma)\exp(\int_{0}^x m(u)du), & x \geq 0, \\
\end{cases}
\end{equation}
with $C(\gamma)=\Big(\int_0^\infty\exp(\int_{0}^x m(u)du)dx +\frac{\Phi(\gamma)}{\phi(\gamma)}\Big)^{-1}$. Moreover,
\begin{equation}
\int_0^\infty \omega(x) dx = \frac{ {\cal L}(\gamma) }{ {\cal L}(\gamma) + \frac{\Phi(\gamma)}{\phi(\gamma)} }. \label{eqn:Probability_that_Xt_is_positive}
\end{equation}
\end{prop}
\begin{proof}
Since the diffusion process $X(t)$ has piecewise continuous parameters, we can apply the procedure developed in \cite{brownewhitt} to find the stationary distribution. This procedure consists of composing the density function as in \refEquation{222} based on the density function of a diffusion process with drift $-\gamma-a(x)$ for $x>0$ and of an \gls{OU} process with drift $-\gamma-x$ for $x<0$. The function $C(\gamma)$ normalizes the distribution.

Equation \refEquation{eqn:Probability_that_Xt_is_positive} follows after substituting \refEquation{eqn:Drift_of_the_diffusion_process} with $a(x)= - f'(x) / f(x)$ into \refEquation{222} and evaluating
\begin{equation}
\int_{0}^\infty \exp\Big(\int_{0}^x m(u)du\Big)dx=\int_0^{\infty}\,e^{-\gamma x}\,f(x)\,dx={\cal L}(\gamma),
\end{equation}
proving \refEquation{eqn:Probability_that_Xt_is_positive}.
\end{proof}


A natural approach now is to approximate the distribution of $\vC{Q}{s}(t)$ by the distribution of $s + X(t) \sqrt{s}$ when $s$ is large.
In \refSection{sec4}, specifically \refTheorem{thm4.10}, we show that \refEquation{eqn:Probability_that_Xt_is_positive} equals $\lim_{s \rightarrow \infty} \vC{D}{s}( \rho )$, as expected, and we also derive the most relevant correction terms for finite $s$. Important here is that $D_s(\rho)$ converges to a value in the interval $(0,1)$ as $s\pr\infty$, which confirms that the local control in \refEquation{1.2} leads to a non-degenerate limit. In \cite{ref2}, $s$-independent control policies have been considered for which $D_s(\rho)$ has $1 / \sqrt{s}$-behavior for large $s$.

\section{\glsentrytext{QED} approximations for global control} \label{sec4}

In this section we focus on global control defined by the scaling profile $f$ and Ansatz (\ref{2.11}).
For this type of control there is a convenient manner of approximating $F_s(\rho)$ as $s\pr\infty$ in terms of the Laplace transform of $f$. Define
\begin{equation} \label{2.16}
\gamma_s = {-}\sqrt{s} \ln{ (1-\gamma/\sqrt{s}) }, \quad \gamma\in \dR.
\end{equation}
Utilizing \refEquation{2.10} and \refEquation{2.11} and recalling that $\rho=1-\gamma/\sqrt{s}$, we can write \refEquation{2.3} as
\begin{equation} \label{2.17}
F_s(\rho)=\sum_{n=0}^{\infty}\,e^{-\frac{n+1}{\sqrt{s}}\gamma_s}\,f\Bigl(\frac{n+1}{\sqrt{s}}\Bigr),
\end{equation}
This expression for $F_s(\rho)$ is instrumental for our analysis.
We apply \gls{EM} summation to \refEquation{2.17}, in order to replace the summation over $n$ by an integral and an appropriate number of error terms. The approach is explained in \refSection{ems}, and leads to the \gls{QED} approximations for the stationary delay and rejection probability presented in \refSection{subsec4.1}. These approximations are demonstrated in \refSection{subsec4.2} for several types of global control.

\subsection{\glsentrytext{EM} summation}\label{ems}

We assume that the function $f(x)$ in (\ref{2.11}) is non-negative, non-increasing and smooth, that is $f\in C^4([0,\infty))$, and that $f(0)=1$. Furthermore, we assume that for any $\gamma>\gamma_{\min}$, $e^{-\gamma x}\,f^{(j)}(x)\in L_1([0,\infty))$ and $e^{-\gamma x}\,f^{(j)}(x)\pr0$ as $x\pr\infty$
for $j=0,1,2,3,4$.

We shall use the following form of the \gls{EM} summation formula about which more details are collected in \refAppendixSection{appB}. Assume that $g:[0,\infty)\pr\dR$ with $g\in C^2([0,\infty))$ and $g^{(j)}\in L^1([0,\infty))$, $j=0,1,2$. Then
\begin{equation} \label{4.2}
\sum_{n=0}^{\infty} g\Bigl(\frac{n+1}{\sqrt{s}}\Bigr) = \sqrt{s} \int_{\frac{1}{2\sqrt{s}}}^\infty  g(x) dx+\frac{1}{24\sqrt{s}} g'\Bigl(\frac{1}{2\sqrt{s}}\Bigr) +R,
\end{equation}
where
\begin{equation} \label{4.3}
|R| \leq \frac{1}{12\sqrt{s}} \int_0^{\infty} |g^{(2)}(x)| dx.
\end{equation}
When also $g^{(4)}\in L^1([0,\infty))$, we have the asymptotically tighter bound
\begin{equation} \label{4.4}
|R| \leq \frac{1}{384 s \sqrt{s}} \int_0^{\infty}\,|g^{(4)}(x)| dx.
\end{equation}
By setting $g(x)=e^{-\gamma x}f(x)$ for $x\geq0$ and $\gamma>\gamma_{\min}$, and using these formulas, we will now obtain several QED  approximations.

\subsection{Corrected \glsentrytext{QED} approximations} \label{subsec4.1}

We first present a result for $F_s(\rho)$.

\begin{thm} \label{thm4.1}
With $\rho=1-\gamma/\sqrt{s}$ and $\gamma_{\min}<\gamma\leq\sqrt{s}$,
\begin{equation} \label{4.5}
F_s(\rho) = \sqrt{s} {\cal L}(\gamma_s) - \frac12 + O\Bigl( \frac{1}{\sqrt{s}} \Bigr),
\end{equation}
where $\gamma_s={-}\sqrt{s} \ln{ (1-\gamma/\sqrt{s}) }$ and where $O(1/\sqrt{s})$ holds uniformly in any compact set of $\gamma$'s contained in $(\gamma_{\min},\infty)$.
\end{thm}

\begin{proof} We have from (\ref{2.17}), (\ref{4.2}) and (\ref{4.3}) that
\begin{equation} \label{4.6}
F_s(\rho) = \sqrt{s} \int_{\frac{1}{2\sqrt{s}}}^{\infty} e^{-\gamma_sx} f(x) dx + \frac{1}{24\sqrt{s}} (e^{-\gamma_sx} f(x))' \Bigl( \frac{1}{2\sqrt{s}} \Bigr) + R
\end{equation}
with
\begin{equation} \label{4.7}
|R| \leq \frac{1}{12\sqrt{s}} \int_0^{\infty} | (e^{-\gamma_sx} f(x))^{(2)}(x) | dx.
\end{equation}
Assume that $\gamma$ is restricted to a compact set $C$ contained in $(\gamma_{\min},\infty)$. Then $\gamma_s$ is restricted to a compact set $D$ contained in $(\gamma_{\min},\infty)$ for all $s \geq 1$ with $\sqrt{s} \geq 2 \max\{ | \gamma | \ | \ \gamma \in C \}$. Hence
\begin{equation} \label{4.8}
e^{-\gamma_sx} f(x) - 1 = O\Bigl(\frac{1}{\sqrt{s}}\Bigr), \quad 0 \leq x \leq \frac{1}{2\sqrt{s}},
\end{equation}
where $O(1/\sqrt{s})$ holds uniformly in $\gamma\in C$. Therefore, we can replace the integral at the right-hand side of (\ref{4.6}) by ${\cal L}(\gamma_s)-1/(2\sqrt{s})$, at the expense of an error $O(1/s)$ uniformly in $\gamma\in C$. Furthermore,
\begin{equation} \label{4.9}
(e^{-\gamma_sx} f(x))' \Bigl(\frac{1}{2\sqrt{s}}\Bigr) = O(1),\quad s \geq 1,
\end{equation}
uniformly in $\gamma\in C$ by smoothness of $f$. Finally, $R = O( 1/\sqrt{s} )$ uniformly in $\gamma \in C$ since there is the bound
\begin{equation} \label{4.10}
|R| \leq \frac{1}{12\sqrt{s}} \int_0^{\infty} e^{-\gamma_sx} ( |\gamma_s|^2 + 2 |\gamma_s| |f'(x)|+|f''(x)| ) dx
\end{equation}
in which $\gamma_s \in D$ with $f$ satisfying the assumptions made at the beginning of \refSection{ems}.
\end{proof}

\refTheorem{thm4.1} yields a simple and often accurate approximation of $F_s(\rho)$, which we illustrate using examples in \refSection{subsec4.2}. However, in the leading term $\sqrt{s}\,{\cal L}(\gamma_s)$, the dependence on the number of servers $s$ and the parameter $\gamma$ is combined into the single quantity $\gamma_s$.
 A more insightful result is given in \refTheorem{thm4.2} below, where the dependence of the approximating terms on $s$ and $\gamma$ is separated.

\begin{thm} \label{thm4.2}
With $\rho=1-\gamma/\sqrt{s}$ and $\gamma_{\min}<\gamma\leq\sqrt{s}$,
\begin{equation} \label{4.11}
F_s(\rho)=\sqrt{s}\,{\cal L}(\gamma)+{\cal M}(\gamma)+O\Bigl(\frac{1}{\sqrt{s}}\Bigr),
\end{equation}
where
\begin{equation} \label{4.12}
{\cal M}(\gamma)=\frac12\gamma^2\,{\cal L}'(\gamma)-\frac12,
\end{equation}
and where $O(1/\sqrt{s})$ holds uniformly in any compact set of $\gamma$'s contained in $(\gamma_{\min},\infty)$. In leading order, the $O(1/\sqrt{s})$ is given as ${\cal N}(\gamma)/\sqrt{s}$, where
\begin{equation} \label{4.13}
{\cal N}(\gamma)=\frac13 \gamma^3\,{\cal L}'(\gamma)+\frac18 \gamma^4\,{\cal L}''(\gamma)+\frac{1}{12} (\gamma-f'(0)).
\end{equation}
\end{thm}

\begin{proof}
This result is obtained from (\ref{4.2}) in a similar way as \refTheorem{thm4.1}, using now the estimate (\ref{4.4}) of $R$, and approximating $\int_0^{1/(2\sqrt{s})} e^{-\gamma_sx} f(x) dx$ and $(e^{-\gamma_sx} f(x))' ( 1/(2\sqrt{s}) )$
more carefully. In particular, we have
\begin{equation} \label{4.15}
\int_0^{1/(2\sqrt{s})} e^{-\gamma_sx} f(x) dx = \frac{1}{2\sqrt{s}} + \frac{1}{8s} (f'(0)-\gamma_s) + O\Bigl( \frac{1}{s\sqrt{s}} \Bigr)
\end{equation}
and
\begin{equation} \label{4.16}
\frac{d}{dx}\,(e^{-\gamma_sx}\,f(x))\,\Bigl(\frac{1}{2\sqrt{s}}\Bigr)=\gamma_s-f'(0)+O\Bigl(\frac{1}{\sqrt{s}}\Bigr).
\end{equation}
Furthermore, because $\gamma_s = \gamma + \gamma^2 / ( 2 \sqrt{s} ) + \gamma^3 / (3s) + \ldots$ for $| \gamma | < \sqrt{s}$, we can approximate ${\cal L}(\gamma_s)$ by
\begin{align} \label{4.17}
{\cal L}(\gamma_s) - {\cal L}(\gamma) &=~(\gamma_s-\gamma){\cal L}'(\gamma)+\frac12 (\gamma_s-\gamma)^2\,{\cal L}''(\gamma)+\frac16 (\gamma_s-\gamma)^3\, {\cal L}'''(\gamma)+ \ldots \nonumber\\
&=~\frac{\gamma^2}{2\sqrt{s}}\,{\cal L}'(\gamma)+\frac1s\Big(\frac13 \gamma^3{\cal L}'(\gamma)+\frac18 \gamma^4\,{\cal L}''(\gamma)\Big)+O\Bigl(\frac{1}{s\sqrt{s}}\Bigr).
\end{align}
The $O$'s in (\ref{4.15})--(\ref{4.17}) hold uniformly in $\gamma$ in any compact set contained in $(\gamma_{\min},\infty)$.
\end{proof}

We use the following short-hand notations for approximations of $F_s(\rho)$ and $B_s(\rho)$ as they occur in the performance measures $D_s(\rho)$ and $D_s^R(\rho)$ in \refEquation{2.22} and \refEquation{2.23}. We write \refEquation{4.11} using \refEquation{4.13} as
\begin{equation} \label{4.33}
F_s(\rho)=\sqrt{s}\,{\cal L}+{\cal M}+\frac{1}{\sqrt{s}}\,{\cal N}=\sqrt{s}\,{\cal L}+{\cal M}+O\Bigl( \frac{1}{\sqrt{s}}\Bigr),
\end{equation}
and we write the Jagerman approximation of $B_s(\rho)$, see \cite{ref2}
and \cite[Theorem~14]{ref4}, as
\begin{equation} \label{4.34}
B_s(\rho)=\frac{1}{\sqrt{s}}\,g+\frac1s\,h+O\Bigl(\frac{1}{s\sqrt{s}}\Bigr).
\end{equation}
Here, $\rho=1-\gamma/\sqrt{s}$, and
\begin{equation} \label{4.35}
g(\gamma) = \frac{\varp(\gamma)}{\Phi(\gamma)}, \quad h(\gamma) = {-}\frac13 \bigl( \gamma^2+(\gamma^2+2) g(\gamma) \bigr) g(\gamma).
\end{equation}

The following results for $D_s(\rho)$ and $D_s^R(\rho)$ are proved in \refAppendixSection{proofthm4.10} using \refEquation{4.33} and \refEquation{4.34}.

\begin{thm}[Corrected \gls{QED} approximations] \label{thm4.10}
The stationary probability of delay satisfies
\begin{equation} \label{4.44}
D_s(\rho)=T_1(\gamma)+\frac{1}{\sqrt{s}}\,T_2(\gamma)+O\Bigl(\frac1s\Bigr),
\end{equation}
where
\begin{equation} \label{4.45}
T_1=\frac{g{\cal L}}{1+g{\cal L}},\quad T_2=\frac{(h+g^2)\,{\cal L}+g({\cal M}+1)}{(1+g{\cal L})^2},
\end{equation}
and where $O(1/s)$ holds uniformly in any compact set of $\gamma$'s contained in $(\gamma_{\min},\infty)$. The stationary rejection probability satisfies
\begin{equation} \label{4.46}
D_s^R(\rho)=\frac{1}{\sqrt{s}}\,T_1^R(\gamma)+\frac1s\,T_2^R(\gamma)+O\Bigl(\frac{1}{s\sqrt{s}}\Bigr),
\end{equation}
where
\begin{equation} \label{4.47}
T_1^R=\frac{1-\gamma{\cal L}}{1+g{\cal L}}\,g,~~~~~T_2^R=\frac{1-\gamma{\cal L}}{1+g{\cal L}}\, \Bigl(h-\gamma g\,\frac{\gamma{\cal L}+{\cal M}}{1-\gamma{\cal L}}-g\,\frac{h{\cal L}+g{\cal M}}{1+g{\cal L}}\Bigr),
\end{equation}
and where $O(1/s\sqrt{s})$ holds uniformly in any compact set of $\gamma$'s contained in $(\gamma_{\min},\infty)$.
\end{thm}

\subsection{Examples} \label{subsec4.2}

We now present several examples to illustrate Theorems~\ref{thm4.1}, \ref{thm4.2} and \ref{thm4.10}.


\subsubsection{Modified-drift control (global)} \label{exam4.6}

Consider $f(x) = p^x$ for $x \geq 0$, with $p \in (0,1)$ fixed. Then, $\gamma_{\min} = \ln{p}$, $P_s = p^{1/\sqrt{s}}$ and
\begin{equation} \label{4.24}
F_s(\rho) = \frac{p^{1/\sqrt{s}}(1-\gamma/\sqrt{s})}{1-p^{1/\sqrt{s}}(1-\gamma/\sqrt{s})}, \quad \sqrt{s}(1-p^{-1/\sqrt{s}}) < \gamma \leq \sqrt{s}.
\end{equation}
\refTheorem{thm4.2} gives the approximation
\begin{equation} \label{4.25}
F_s(\rho) \approx \frac{\sqrt{s}}{\gamma- \ln{p} }-\frac{\gamma^2}{2( \gamma- \ln{p} )^2}-\frac12,\quad \gamma > \ln{p}.
\end{equation}

\subsubsection{Erlang A control (global)} \label{exam4.7}

Let $f(x) = p^{x^2}$ for $x \geq 0$, with $p \in (0,1)$ fixed. In this case, $\gamma_{\min} = {-}\infty$ and $P_s = 0$. Also,
\begin{equation} \label{4.29}
{\cal L}(\gamma) = \frac{1}{\sqrt{\alpha}} \MM( \gamma / (2\sqrt{\alpha}) ), \quad \gamma \in \dR,
\end{equation}
where $\alpha={-}\ln{p}$ and $\MM$ is Mills' ratio, defined as
$\MM(\delta)=e^{\delta^2} \int_{\delta}^{\infty} e^{-y^2} dy$ for $\delta \in \dR$ \cite[\S 7.8]{OLBC10}. Taking the derivative, we find that
\begin{equation} \label{4.31}
{\cal L}'(\gamma) = \frac{\gamma}{2\alpha\sqrt{\alpha}} \MM( \gamma / (2\sqrt{\alpha}) ) - \frac{1}{2\alpha}, \quad \gamma \in \dR,
\end{equation}
and we then obtain from \refTheorem{thm4.2} the approximation
\begin{equation} \label{4.32}
F_s(\rho) \approx \frac{\sqrt{s}}{\sqrt{\alpha}} \MM( \gamma / (2\sqrt{\alpha}) ) + \frac14 \Bigl(\frac{\gamma}{\sqrt{\alpha}}\Bigr)^3 \MM(\gamma/ (2\sqrt{\alpha}) ) - \frac14 \Bigl( \frac{\gamma}{\sqrt{\alpha}} \Bigr)^2 - \frac12, \quad \gamma \in \dR.
\end{equation}
Note that the approximation $F_s(\rho) \approx \sqrt{s} {\cal L} (\gamma_s)-1/2$ as described in \refTheorem{thm4.1} can also be computed from \refEquation{4.29} after recalling that $\gamma_s = {-} \sqrt{s} \ln{ (1-\gamma/\sqrt{s}) }$.

\subsubsection{Scaled buffer control (global)}

Take a fixed $\eta>0$ and 
set $p_s(k) = \indicator{ k+1 < \eta \sqrt{s} }$ for $k \in \naturalNumbersZero$.
Thus for $n \in \naturalNumbersZero$, $q_s(n) = p_s(n) = f( (n+1) /\sqrt{s} )$,
with $f(x) = \indicator{ x \in [0, \eta) }$ for $x \geq 0$.
It follows that $P_s = 0$, $\gamma_{\min} = {-}\infty$ and
\begin{equation} \label{5.4}
F_s(\rho) = \Bigl( \frac{\sqrt{s}}{\gamma} - 1 \Bigr) \Bigl( 1-\Bigl( 1-\frac{\gamma}{\sqrt{s}} \Bigr)^{\lfloor\eta\sqrt{s}\rfloor} \Bigr), \quad -\infty < \gamma \leq \sqrt{s}.
\end{equation}

The function $f$ is not smooth, and strictly speaking, Theorems~\ref{thm4.1} and \ref{thm4.2} do not apply. Still, we can calculate
\begin{equation} \label{5.5}
{\cal L}(\gamma) = \frac{1}{\gamma} (1-e^{-\gamma\eta}), \quad {\cal L}'(\gamma) = - \frac{1}{\gamma^2} (1-(1+\gamma\eta) e^{-\gamma\eta}), \quad \gamma \in \dR,
\end{equation}
and use the approximation that \refTheorem{thm4.1} would give, i.e.\
\begin{equation} \label{5.6}
F_s(\rho) \approx \sqrt{s} {\cal L}(\gamma_s) - \frac{1}{2} =\frac{1-(1-\gamma/\sqrt{s})^{\eta\sqrt{s}}}{ -\ln{ (1-\gamma/\sqrt{s})} }- \frac{1}{2}.
\end{equation}
Alternatively, \refTheorem{thm4.2} gives the approximation
\begin{equation} \label{5.7}
F_s(\rho) \approx \sqrt{s}{\cal L}(\gamma) + \frac12\gamma^2{\cal L}'(\gamma) - \frac12 = \Bigl(\frac{\sqrt{s}}{\gamma}-1\Bigr) (1-e^{-\gamma\eta}) - \frac12 e^{-\gamma\eta}(1-\gamma\eta), \quad \gamma \in \dR.
\end{equation}

While \refEquation{5.4} has a jump as a function of $s$ at all $s$ where $\eta\sqrt{s}$ is integer, its approximations in \refEquation{5.6} and \refEquation{5.7} are smooth functions of $s$, if we consider $s \geq 1$ as a continuous variable. The averages of the approximations over $s$-intervals $[(k/\eta)^2,((k+1)/\eta)^2]$ with integer $k$ agree well with the average of \refEquation{5.4} over these intervals. Thus, while Theorems~\ref{thm4.1} and \ref{thm4.2} do not apply, they yield approximations that perform well in an appropriate average sense. This is also illustrated in \refFigure{fig:Ds_as_function_of_eta}.

\begin{figure}[!hbtp]
\begin{center}
\begin{tikzpicture}
\node[anchor=south west,inner sep=0] at (0,0) {\includegraphics[width=3.3in, keepaspectratio]{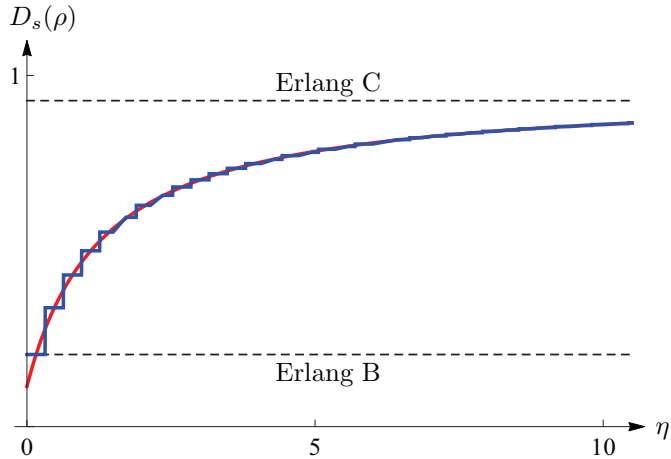}};
\node[anchor=south west,inner sep=0] at (8.5,0.3) {$\eta$};
\node[anchor=south west,inner sep=0] at (0,5.7) {$D_s(\rho)$};
\node[anchor=south west,inner sep=0] at (3.5,1.0) {Erlang B};
\node[anchor=south west,inner sep=0] at (3.5,4.85) {Erlang C};
\end{tikzpicture}
\caption{The stationary probability of delay for $s = 10$ and scaled buffer control. The jagged, blue curve gives the exact value \refEquation{5.4} and the red, smooth curve pertains to approximation \refEquation{5.7}.}
\label{fig:Ds_as_function_of_eta}
\end{center}
\end{figure}


\section{\glsentrytext{QED} approximations for local control} \label{sec5}

A clear technical advantage of global control is that it leads to infinite-series expressions for the performance measures $D_s(\rho)$ and $D_s^R(\rho)$ that are directly amenable to asymptotic analysis based on \gls{EM} summation. This approach was followed in \refSection{sec4} and led to \refTheorem{thm4.10}. As argued in \refSection{sec3}, in some practical cases it is more natural to work with the local control defined
in \refEquation{1.2}. In this section we show that \refTheorem{thm4.10} also gives sharp approximations for local control. Indeed, in \refSection{conn} it was argued that for local control,
$q_s(n)\approx f((n+1)/\sqrt{s})$
with $f$ defined as in \refEquation{6.5}. Consider for instance the  modified-drift control
in \refEquation{mdc}, in which case 
\begin{align} \label{m1}
F_s(\rho)&=\sum_{n=0}^\infty p^{\frac{n+1}{\sqrt{s}}}\Big(1-\frac{\gamma}{\sqrt{s}}\Big)^{n+1}\nonumber\\
&=\frac{\sqrt{s}}{\gamma-\ln p}-\frac{\gamma}{\gamma-\ln p}-\frac{1}{2}\Big(\frac{\ln p}{\gamma-\ln p}\Big)^2+O\Bigl(\frac{1}{\sqrt{s}}\Bigr).
\end{align}
Here, the second equality follows from the \gls{QED} approximation in \refEquation{4.11}. The local counterpart follows from $a(x)=-f'(x)/f(x)=-\ln p$ and \refEquation{1.2}, for which
\begin{equation} \label{m2}
F_s(\rho) = \sum_{n=0}^\infty \Big(\frac{1}{1-\frac{1}{\sqrt{s}}\ln p}\Big)^{n+1}\Big(1-\frac{\gamma}{\sqrt{s}}\Big)^{n+1}=\frac{\sqrt{s}}{\gamma-\ln p}-\frac{\gamma}{\gamma-\ln p}.
\end{equation}
The second equality in \refEquation{m2} follows from summation of a geometric series. Hence, for the example of modified-drift control, it can be seen from the close resemblance of the last members of \refEquation{m1} and \refEquation{m2} that approximating local control by global control yields sharp estimates in the \gls{QED} regime.

In \refSection{logl} we make formal the accuracy of the approximation $q_s(n)\approx f((n+1)/\sqrt{s})$ for a wide range of local controls. As it turns out, the approximation becomes asymptotically correct in the \gls{QED} regime. Therefore, for local controls for which the Ansatz $q_s(n)=f((n+1)/\sqrt{s})$ does not hold precisely, it will give sharp approximations for the performance measures in the \gls{QED} regime.
For the example of Erlang A control, with $a(x)=\vart x$, this is demonstrated in \refSection{ErlangAAA}.

%

\subsection{Approximating local by global control} \label{logl}

We first present a general result for all functions $a$ considered in this paper. The proof is given in \refAppendixSection{proofpropallA}.
\begin{prop}[Relation between global and local control] \label{propallA}
Assume that $a(x)$ is non-negative and non-decreasing in $x\geq 0$. There is an increasing function $\psi(s)$ of $s\geq 1$ with $\psi(s)\to\infty$, $s\to\infty$, such that
\begin{equation}
q_s(n) = f\Bigl(\frac{n+1}{\sqrt{s}}\Bigr)\Big(1+O(s^{- \frac{1}{4} })\Big), \quad 0 \leq n+1 \leq \sqrt{s} \psi(s),
\end{equation}
where $f$ is given as in \refEquation{6.5}.
\end{prop}

We next illustrate \refProposition{propallA} for a special case of increasing $a$.
Let $\vart>0$ and $\alpha\geq0$, and let $a(x)=\vart\,x^{\alpha}$ for $x \geq 0$.
Inspecting the proof of \refProposition{propallA}, case $\delta=1/4$,
it is seen that $\psi$ is found by requiring
\begin{equation}
a\Bigl(\frac{n+1}{\sqrt{s}}\Bigr)\leq \frac12 s^{ \frac{1}{4} }, \quad \int_0^{\frac{n+1}{\sqrt{s}}}a^2(x)dx\leq s^{ \frac{1}{4} }.
\end{equation}
\refProposition{propallA} yields
\begin{equation}
q_s(n) = \exp\Bigl(\frac{-\vart}{\alpha+1}\Big(\frac{n+1}{\sqrt{s}}\Big)^{\alpha+1}\Bigr)\Big(1+O(s^{- \frac{1}{4} })\Big),\quad 0\leq n+1\leq \sqrt{s} \psi(s),
\end{equation}
with
\begin{equation} \label{pssss}
\psi(s) = \min{ \Big\{ \Big(\frac{1}{2\vart}\Big)^{\frac{1}{\alpha}}s^{\frac{1}{4\alpha}},\Big(\frac{2\alpha+1}{\vart^2}\Big)^{\frac{1}{2\alpha+1}}s^{\frac{1}{4(2\alpha+1)}} \Big\} }.
\end{equation}

\subsection{Erlang A control (local)} \label{ErlangAAA}

We now consider in detail Erlang A control, in order to demonstrate our obtained QED approximations. Erlang A control gives rise to a birth--death process that is identical to the classical Erlang A model \cite{garnett,erlanga}. It allows us to express $F_s(\rho)$ in terms of the confluent hypergeometric function and to subsequently show that an asymptotic expansion of the confluent hypergeometric function leads to a QED approximation similar to \refEquation{4.32}.

We thus consider the example
\begin{equation} \label{6.12}
p_s(k)=\frac{1}{1+(k+1)\,\frac{\vart}{s}}, \quad k \in \naturalNumbersZero,
\end{equation}
which corresponds to $\alpha=1$ when $a(x) = \vart x^\alpha$ for $x \geq 0$, so $f(x) = \exp{ ( {-} \vart x^2 / 2 ) }$ for $x \geq 0$.

\begin{prop} \label{prop6.3}
Assume that $p_s(k)$ is given by \refEquation{6.12}. Then
\begin{equation} \label{6.14}
F_s(\rho) = \frac{1}{\rho} \bigl( M( 1, s/\vart, s\rho/\vart ) - 1 - \rho \bigr), \quad \rho \geq 0,
\end{equation}
in which
\begin{equation} \label{6.15}
M(a,b,z) = \sum_{n=0}^{\infty} \frac{(a)_n}{(b)_n} \frac{z^n}{n!},\quad z \in \dC
\end{equation}
is the confluent hypergeometric function {\rm \cite[Chapter 13]{OLBC10}}, with $(x)_l$ the Pochhammer symbol, i.e.~$(x)_l = 0$ for $l = 1$ and $(x)_l = x (x+1) \cdot \ldots \cdot (x+l-1)$ for $l \geq 1$.
\end{prop}

\begin{proof}
For $n \in \naturalNumbersZero$,
\begin{equation} \label{6.16}
q_s(n)=\prod_{k=0}^n\:\frac{1}{1+(k+1)\,\frac{\vart}{s}}=\frac{(s/\vart)^{n+2}}{(s/\vart)_{n+2}}.
\end{equation}
Therefore
\begin{equation} \label{6.18}
F_s(\rho)=\sum_{n=0}^{\infty}\,q_s(n)\,\rho^{n+1}=\sum_{n=0}^{\infty}\:\frac{(s/\vart)^{n+2}}{(s/\vart)_{n+2}}\,\rho^{n+1},
\end{equation}
and (\ref{6.14}) follows after some rearrangements. \end{proof}

In \cite[13.8(ii)]{OLBC10} the asymptotics of $M(a,b,z)$ is considered when $b$ and $z$ are large while $a$ is fixed and $b/z$ is in a compact set contained in $(0,\infty)$. With
\begin{equation} \label{6.19}
a=1\,,~~b=s/\vart\,,~~z=s\rho/\vart
\end{equation}
and $s\pr\infty$, while $\rho=1-\gamma/\sqrt{s}$ is close to 1, this is precisely the situation we are interested in. Temme \cite{ref5} gives a complete asymptotic series, and this leads to the following result.

\begin{prop} \label{prop6.4}
As $s\pr\infty$,
\begin{equation} \label{6.23}
F_s(\rho)\sim\Bigl(\frac{2}{\vart}\Bigr)^{ \frac{1}{2} } \MM(\gamma/\sqrt{2\vart}) \sqrt{s}+\frac{\gamma^3\sqrt{2}}{3\vart^{ \frac{3}{2} }}\MM(\gamma/\sqrt{2\vart}) - \frac{\gamma^2}{3\vart} - \frac{2}{3}.
\end{equation}
\end{prop}

\noindent
\begin{proof}
The first two terms of Temme's asymptotic series \cite{ref5} are as follows. Let $\zeta=\sqrt{2(\rho-1-{\rm ln}\,\rho)}$,
where ${\rm sgn}(\zeta)={\rm sgn}(\rho-1)$. Then
\begin{align} \label{6.21}
\hspace*{-6mm}M\Bigl(1,\frac{s}{\vart},\frac{s\rho}{\vart}\Bigr) \sim \Bigl(\frac{s}{\vart}\Bigr)^{ \frac{1}{2} } \exp\Bigl(\frac{\zeta^2s}{4\vart}\Bigr)
\Bigl\{ &\rho U\Bigl(\frac12,{-}\zeta\Bigl( \frac{s}{\vart}\Bigr)^{ \frac{1}{2} }\Bigr) \nonumber \\
&+ \Bigl(\rho-\frac{\zeta}{\rho-1}\Bigr)\,\frac{1}{\zeta\Bigl(\frac{s}{\vart}\Bigr)^{ \frac{1}{2} }}\,U\Bigl({-}\frac12,{-}\zeta\Bigl(\frac{s}{\vart}\Bigr)^{ \frac{1}{2} }\Bigr)\Bigr\}
\end{align}
with $U$ the parabolic cylinder function of \cite[Ch.~12]{OLBC10}. In this particular case \cite[\S 12.5.1, \S 12.7.1]{OLBC10},
\begin{equation} \label{6.22}
U \Bigl( \frac{1}{2},z \Bigr) = \e^{-\frac14 z^2} \int_0^{\infty} \e^{-\frac12 t^2- zt} dt, \quad U \Bigl( {-} \frac{1}{2}, z \Bigr) = \e^{-\frac14 z^2}.
\end{equation}
Since $\rho=1-\gamma/\sqrt{s}$, 
\begin{equation} \label{6.24}
\zeta = {-} \frac{\gamma}{\sqrt{s}}-\frac{\gamma^2}{3s} + O\Bigl(\frac{\gamma^3}{s\sqrt{s}}\Bigr), \quad \rho-\frac{\zeta}{\rho-1} = {-} \frac43 \frac{\gamma}{\sqrt{s}}+O \Bigl(\frac{\gamma^2}{s} \Bigr).
\end{equation}
Substituting in \refEquation{6.21}, together with \refEquation{6.22}, we get
\begin{align}
F_s(\rho) \sim \, {-} 1 - \frac{1}{\rho}+\Bigl(\frac{s}{\vart}\Bigr)^{ \frac{1}{2} } \exp\Bigl(\frac{\zeta^2s}{4\vart}\Bigr) \Bigl\{ &\exp\Bigl({-} \frac{\zeta^2s}{4\vart}\Bigr) \int_0^{\infty}
\e^{-\frac12 t^2+\zeta (\frac{s}{\vart})^{ \frac{1}{2} }t} dt \nonumber \\
&+ \frac{\rho-\zeta/(\rho-1)}{\zeta\rho(s/\vart)^{ \frac{1}{2} }} \exp\Bigl({-} \frac{\zeta^2s}{4\vart}\Bigr)\Bigr\},
\end{align}
so that
\begin{equation} \label{6.25}
F_s(\rho) \sim -\frac23 + \Bigl(\frac{s}{\vart}\Bigr)^{ \frac{1}{2} } \int_0^{\infty} \e^{-\frac12 t^2+\zeta(\frac{s}{\vart})^{ \frac{1}{2} }t} dt+O\Bigl(\frac{\gamma}{\sqrt{s}}\Bigr).
\end{equation}
Next
\begin{equation} \label{6.26}
\zeta\Bigl(\frac{s}{\vart}\Bigr)^{ \frac{1}{2} } = {-} \frac{\gamma}{\sqrt{\vart}}-\frac{\gamma^2}{3\sqrt{\vart s}}+O\Bigl(\frac{\gamma^3}{s}\Bigr),
\end{equation}
and using this in \refEquation{6.25}, we get
\begin{align} \label{6.27}
F_s(\rho) =& - \frac23 + \Bigl(\frac{s}{\vart}\Bigr)^{ \frac{1}{2} } \int_0^{\infty} \e^{ -\frac12 t^2-\frac{\gamma t}{\sqrt{\vart}}} dt \nonumber \\
&- \frac13 \Bigl(\frac{s}{\vart}\Bigr)^{ \frac{1}{2} } \frac{\gamma^2}{\sqrt{\vart s}} \int_0^{\infty} \e^{-\frac12 t^2-\frac{\gamma t}{\sqrt{\vart}}} t dt+O\Bigl(\frac{\gamma}{\sqrt{s}} \Bigr).
\end{align}
Finally, using that
\begin{gather}
\label{6.28}
\int_0^{\infty} \e^{-\frac12 t^2-\frac{\gamma t}{\sqrt{\vart}}} dt = \sqrt{2} \MM(\gamma/\sqrt{2\vart}),
\\
\label{6.29}
\int_0^{\infty} \e^{-\frac12 t^2-\frac{\gamma t}{\sqrt{\vart}}} t dt = {-}\MM'(\gamma/\sqrt{2\vart}) = {-} \Bigl( \frac{2\gamma}{\sqrt{2\vart}} \MM(\gamma/\sqrt{2\vart})-1 \Bigr),
\end{gather}
we obtain the result.
\end{proof}

It is instructive to rewrite the asymptotics \refEquation{4.32} of $F_s(\rho)$ in Example~\ref{exam4.7} for the case that $q_n = f((n+1)/\sqrt{s})$ and $f(x) = p^x$, in terms of $\vart = 2\alpha = - 2\ln p$. In doing so, \refEquation{4.32} becomes
\begin{equation} \label{5.22}
F_s(\rho) \sim \Bigl(\frac{2}{\vart}\Bigr)^{1/2} \MM(\gamma/\sqrt{2\vart}) \sqrt{s}+\frac{\gamma^3}{\sqrt{2}\vart^{3/2}}\MM(\gamma/\sqrt{2\vart}) - \frac{\gamma^2}{2\vart} - \frac12.
\end{equation}
Observe the close resemblance between \refEquation{5.22} and \refEquation{6.23}.

\subsubsection{Numerical comparison}

For Erlang A control, for which $f(x) = p^{x^2}$ and $- \ln p = \alpha = \vart / 2$, we have now determined an exact expression and an asymptotic expression for $F_s(\rho)$, given in \refProposition{prop6.3} and \refProposition{prop6.4}, respectively. Through \refEquation{2.22} we then obtain exact and asymptotic expressions for $D_s$. Furthermore, we can obtain approximate values for $D_s$ using the first-order and second-order approximation in \refTheorem{thm4.10}. \refTable{tab:Numerical_comparison_of_different_expressions_of_Ds} shows a numerical comparison when using these different expressions for $D_s$.

\begin{table}[!hbtp]
\small
\begin{center}
\begin{tabular}{cp{0.01pt}cccp{0.01pt}cccp{0.01pt}ccc}
\toprule
& & \multicolumn{3}{c}{$\vart = 1$} & & \multicolumn{3}{c}{$\vart = 10$} & \, & \multicolumn{3}{c}{$\vart = 100$} \\
\cmidrule{3-5} \cmidrule{7-9}  \cmidrule{11-13}
$s$ & \, & $D_s^\textrm{exact}$ & $D_s^\textrm{asymp}$ & $D_s^{\textrm{approx}}$ & & $D_s^\textrm{exact}$ & $D_s^\textrm{asymp}$ & $D_s^{\textrm{approx}}$& & $D_s^\textrm{exact}$ & $D_s^\textrm{asymp}$ & $D_s^{\textrm{approx}}$ \\
\midrule
   1 & \, & 0.59343 & 0.57277 & 0.62582 & & 0.49415 & 0.39305 & 0.48528 & & 0.47591 & 0.29172 & 0.41076 \\
   2 & \, & 0.55437 & 0.54342 & 0.57730 & & 0.41389 & 0.34704 & 0.40797 & & 0.38093 & 0.23525 & 0.31498 \\
   4 & \, & 0.52652 & 0.52092 & 0.54300 & & 0.35137 & 0.31225 & 0.35330 & & 0.29862 & 0.19283 & 0.24726 \\
   8 & \, & 0.50691 & 0.50410 & 0.51874 & & 0.30732 & 0.28658 & 0.31465 & & 0.23226 & 0.16172 & 0.19938 \\
  16 & \, & 0.49313 & 0.49172 & 0.50158 & & 0.27830 & 0.26792 & 0.28731 & & 0.18229 & 0.13925 & 0.16552 \\
  32 & \, & 0.48343 & 0.48273 & 0.48946 & & 0.25956 & 0.25448 & 0.26798 & & 0.14717 & 0.12315 & 0.14157 \\
  64 & \, & 0.47660 & 0.47625 & 0.48088 & & 0.24735 & 0.24487 & 0.25432 & & 0.12407 & 0.11169 & 0.12464 \\
 128 & \, & 0.47178 & 0.47160 & 0.47481 & & 0.23924 & 0.23802 & 0.24465 & & 0.10961 & 0.10354 & 0.11267 \\
 256 & \, & 0.46837 & 0.46828 & 0.47053 & & 0.23375 & 0.23316 & 0.23782 & & 0.10068 & 0.09776 & 0.10421 \\
 512 & \, & 0.46597 & 0.46592 & 0.46749 & & 0.23000 & 0.22970 & 0.23299 & & 0.09506 & 0.09367 & 0.09822 \\
1024 & \, & 0.46427 & 0.46425 & 0.46535 & & 0.22740 & 0.22725 & 0.22957 & & 0.09146 & 0.09774 & 0.09399 \\
\bottomrule
\end{tabular}
\caption{Numerical comparison of different expressions of $D_s$ for $f(x) = \exp{ ( - \vart x^2 / 2 ) }$ and $\gamma = 0.1$. Here, $D_s^{\textrm{exact}}$ is calculated using \refProposition{prop6.3}, $D_s^{\textrm{asymp}}$ using \refProposition{prop6.4} and $D_s^{\textrm{approx}} = T_1 + T_2 / \sqrt{s}$ using \refTheorem{thm4.10}. Furthermore, for all $s$, $T_1 \approx 0.46017, 0.22132$ and $0.08377$ for $\vart = 1, 10$ and $100$, respectively.}
\label{tab:Numerical_comparison_of_different_expressions_of_Ds}
\end{center}
\end{table}

From \refTable{tab:Numerical_comparison_of_different_expressions_of_Ds} we see that the precision of all approximations increase with $s$. We also see that the second-order approximation $T_1 + T_2 / \sqrt{s}$ is more accurate than the first-order approximation $T_1$, particularly for moderate values of $s$.

\section{Conclusions and outlook} \label{sec6}

We have introduced \gls{QED} scaled control, designed to reduce the incoming traffic in periods of congestion, in such a way that the controlled many-server system remains within the domain of attraction of the favorable \gls{QED} regime. The scaled control is chosen such that it affects the typical $O(\sqrt{s})$ queue lengths that arise in the \gls{QED} regime. The class of many-server systems with \gls{QED} control introduced in this paper contains the Erlang B, C and A models as special cases. For all cases we have derived so-called corrected \gls{QED} approximations, which not only identify the \gls{QED} limits as leading terms, but also provide corrections through higher order terms for finite system sizes $s$. As a key example we took the stationary probability of delay, for which the corrected \gls{QED} approximation reads $D_s(\rho)\approx T_1(\gamma)+T_2(\gamma)/\sqrt{s}$, as stated in \refTheorem{thm4.10}. The technique developed in this paper to obtain the corrected diffusion approximations can be easily applied to other characteristics of the stationary distribution, such as the mean and the cumulative distribution function.

Our corrected \gls{QED} approximations pave the way for obtaining optimality results for dimensioning systems \cite{borst}. Consider for instance the basic problem of determining the largest load $\rho$ such that $D_s(\rho)\leq\eps$ with $\eps\in (0,1)$. The delay probability is a function of the two model parameters $s$ and  $\lambda$, and of the control policy. Denote this unique solution by $\rho=\rho_{\rm opt}$ and define $\gamma_{\rm opt}=\sqrt{s}(1-\rho_{\rm opt})$. Asymptotic dimensioning would approximate $D_s(\rho)$ by the \gls{QED} limit $T_1(\gamma)$ that only depends on $\gamma$ (and no longer on both $s$ and  $\lambda$). Hence, the inverse problem can then be approximatively solved by searching for the $\gamma=\gamma_*$ such that $T_1(\gamma)=\eps$, and then setting the load according to $\rho_*=1-\gamma_*/\sqrt{s}$. This procedure is referred to as square-root staffing, and the error $|\gamma_{{\rm opt}}-\gamma_*|$ is called the  {\it optimality gap}. In future work, we will leverage the corrected \gls{QED} approximations derived in the present paper to characterize the optimality gaps for a large class of dimensioning problems.

\section*{Acknowledgments}

This research was financially supported by The Netherlands Organization for Scientific Research (NWO) in the framework of the TOP-GO program and by an ERC Starting Grant.

\bibliographystyle{alpha}
\bibliography{Bibliography}

\begin{thebibliography}{OLBC10}

\bibitem[AM04]{armonymaglaras}
M.~Armony and C.~Maglaras.
\newblock Customer contact centers with a call-back option.
\newblock {\em Oper.\ Res.}, 52:271--292, 2004.

\bibitem[BMR04]{borst}
S.~Borst, A.~Mandelbaum, and M.~Reiman.
\newblock Dimensioning large call centers.
\newblock {\em Oper.\ Res.}, 52:17--34, 2004.

\bibitem[BW95]{brownewhitt}
S.~Browne and W.~Whitt.
\newblock {\em Advances in {Q}ueueing: {T}heory, {M}ethods, and {O}pen
  {P}roblems - Piecewise-linear diffusion processes}.
\newblock CRC Press, Boca Raton, FL, 1995.
\newblock ed.\ J.\ Dshalalow.

\bibitem[Ell98]{ref8}
D.~Elliott.
\newblock The {E}uler-{M}aclaurin formula revisited.
\newblock {\em J.\ Austral.\ Math.\ Soc.\ B}, 40(E):E27--E76, 1998.

\bibitem[FA95]{ref1}
G.~I. Falin and J.~R. Artalejo.
\newblock Approximations for multi-server queues with balking/retrial
  discipline.
\newblock {\em OR Spektrum}, 17:239--244, 1995.

\bibitem[GKM03]{gans}
N.~Gans, G.~Koole, and A.~Mandelbaum.
\newblock Telephone {C}all {C}enters: {T}utorial, {R}eview and {R}esearch
  {P}rospects.
\newblock {\em M\&SOM-Manuf.\ Serv.\ Op.}, 5:79--141, 2003.

\bibitem[GMR02]{garnett}
O.~Garnett, A.~Mandelbaum, and M.~Reiman.
\newblock Designing a call center with impatient customers.
\newblock {\em M\&SOM-Manuf.\ Serv.\ Op.}, 4:208--227, 2002.

\bibitem[Hil56]{ref9}
F.~B. Hildebrand.
\newblock {\em {I}ntroduction to {N}umerical {A}nalysis}.
\newblock McGraw-Hill, New York, USA, 1956.

\bibitem[HW81]{halfinwhitt}
S.~Halfin and W.~Whitt.
\newblock Heavy-traffic limits for queues with many exponential servers.
\newblock {\em Oper.\ Res.}, 29:567--588, 1981.

\bibitem[Igl74]{Iglehart73}
D.~L. Iglehart.
\newblock Weak convergence in applied probability.
\newblock {\em Stoch.\ Proc.\ Appl.}, 2(3):211--241, 1974.

\bibitem[Jag74]{ref4}
D.~Jagerman.
\newblock Some properties of the {E}rlang loss function.
\newblock {\em Bell Syst.\ Tech.\ J.}, 53:525--551, 1974.

\bibitem[JMM04]{jelenkovic}
P.~Jelenkovi\'{c}, A.~Mandelbaum, and P.~Mom\v{c}ilovic.
\newblock Heavy traffic limits for queues with many deterministic servers.
\newblock {\em Queueing~Syst.}, 47:53--69, 2004.

\bibitem[JvL12]{ref2}
A.~J. E.~M. Janssen and J.~S.~H. van Leeuwaarden.
\newblock Staffing many-server systems with admission control and retrials.
\newblock {\em Submitted}, 2012.

\bibitem[JvLZ08]{erlangb}
A.~J. E.~M. Janssen, J.~S.~H. van Leeuwaarden, and B.~Zwart.
\newblock Gaussian expansions and bounds for the {P}oisson distribution applied
  to the {E}rlang {B} formula.
\newblock {\em Adv.\ in Appl.\ Probab.}, 40(1):122--143, 2008.

\bibitem[JvLZ11]{erlangc}
A.~J. E.~M. Janssen, J.~S.~H. van Leeuwaarden, and B.~Zwart.
\newblock Refining square root staffing by expanding {E}rlang {C}.
\newblock {\em Oper.\ Res.}, 59:1512--1522, 2011.

\bibitem[Lyn85]{ref7}
J.~N. Lyness.
\newblock The {E}uler {M}aclaurin expansion for the {C}auchy principal value
  integral.
\newblock {\em Numer.\ Math.}, 46:611--622, 1985.

\bibitem[MM08]{manmom}
A.~Mandelbaum and P.~Mom\v{c}ilovic.
\newblock Queues with many servers: The virtual waiting-time process in the
  \gls{QED} regime.
\newblock {\em Math.\ Oper.\ Res.}, 33:561--586, 2008.

\bibitem[MW04]{masseywallace}
W.~A. Massey and R.~B. Wallace.
\newblock An asymptotically optimal design of the {${M}/{M}/c/k$} queue.
\newblock {\em Unpublished report}, 2004.

\bibitem[MZ04]{maglaraszeevi}
C.~Maglaras and A.~Zeevi.
\newblock Diffusion approximations for a multiclass {M}arkovian service system
  with ``guaranteed'' and ``best-effort'' service levels.
\newblock {\em Math.\ Oper.\ Res.}, 29:786--813, 2004.

\bibitem[OLBC10]{OLBC10}
F.~W.~J. Olver, D.~W. Lozier, R.~F. Boisvert, and C.~W. Clark.
\newblock {\em {NIST} {H}andbook of {M}athematical {F}unctions}.
\newblock Cambridge University Press, Cambridge, United Kingdom, 2010.

\bibitem[Ree09]{reed}
J.~Reed.
\newblock The {${G}$/${GI}$/$N$} queue in the {H}alfin-{W}hitt regime.
\newblock {\em Ann.\ Appl.\ Probab.}, 19:2211--2269, 2009.

\bibitem[Tem78]{ref5}
N.~M. Temme.
\newblock Uniform asymptotic expansions of confluent hypergeometric functions.
\newblock {\em J.\ Inst.\ Math.\ Appl.}, 22:215--223, 1978.

\bibitem[Whi04]{whittrigor}
W.~Whitt.
\newblock Heavy-traffic limits for the {${G}/{H}_2^*/n/m$} queue.
\newblock {\em Math.\ Oper.\ Res.}, 30:1--27, 2004.

\bibitem[Whi05]{whittprox}
W.~Whitt.
\newblock A diffusion approximation for the {${G}/{GI}/n/m$} queue.
\newblock {\em Oper.\ Res.}, 52:922--941, 2005.

\bibitem[ZvLZ12]{erlanga}
B.~Zhang, J.~S.~H. van Leeuwaarden, and B.~Zwart.
\newblock Refining square-root staffing for call centers with impatient
  customers.
\newblock {\em Oper.\ Res.}, 60:461--474, 2012.

\end{thebibliography}

\appendix

\renewcommand\thethm{\Alph{section}.\arabic{thm}}


\section{Proof of \refProposition{prop2.1}} \label{proofprop2.1}

\subsection{Proof of \refProposition{prop2.1}(i)}

We assume that $F_s(\rho)$, $\rho=1-\gamma/\sqrt{s}$, is of the form (\ref{2.3}) with
\begin{equation} \label{3.1}
p_s(0)\cdot \ldots \cdot p_s(n)=q_s(n)=f\Bigl(\frac{n+1}{\sqrt{s}}\Bigr),\quad n=0,1,\ldots ,
\end{equation}
where $s\geq1$ and $f(x)$ is a non-negative and non-increasing function in $x\geq0$ with $f(0)=1$. Furthermore, we recall the definition of $\gamma_s={-}\sqrt{s}\,{\rm ln}(1-\gamma/\sqrt{s})$ in (\ref{2.16}).
The stability result of \refProposition{prop2.1}(i) is a consequence of the following inequality.

\begin{prop} \label{prop3.2}
For $\sqrt{s}(1-\exp({-}\gamma_{\min}/\sqrt{s}))<\gamma\leq0$,
\begin{equation} \label{3.2}
e^{\gamma_s/\sqrt{s}}\,\int_{1/\sqrt{s}}^{\infty}\,e^{-\gamma_sx}\,f(x)\,dx\leq\frac{1}{\sqrt{s}}\, F_s(\rho)\leq e^{-\gamma_s/\sqrt{s}}\,\int_0^{\infty}\,e^{-\gamma_sx}\,f(x)\,dx.
\end{equation}
\end{prop}
\begin{proof}We start by noting that
\begin{equation} \label{3.3}
\gamma>\sqrt{s}(1-e^{-\gamma_{\min}/\sqrt{s}})=:\gamma_{\min,s}\Leftrightarrow\gamma_s>\gamma_{\min}.
\end{equation}
We consider formula (\ref{2.17}) for $F_s(\rho)$. We have for $\gamma\leq0$ and $n=0,1,\ldots $ from monotonicity of $f$ that
\begin{equation} \label{3.4}
f(x)\geq f\Bigl(\frac{n+1}{\sqrt{s}}\Bigr)\,,~~e^{-\gamma_sx}\geq e^{\gamma_s/\sqrt{s}}\,e^{-\frac{n+1}{\sqrt{s}}\gamma_s}\,,~~\frac{n}{\sqrt{s}}\leq x\leq\frac{n+1}{\sqrt{s}}
\end{equation}
and
\begin{equation} \label{3.5}
f(x)\leq f\Bigl(\frac{n+1}{\sqrt{s}}\Bigr)\,,~~e^{-\gamma_sx}\leq e^{-\gamma_s/\sqrt{s}}\,e^{-\frac{n+1}{\sqrt{s}}\gamma_s}\,,~~\frac{n+1}{\sqrt{s}}\leq x\leq\frac{n+2}{\sqrt{s}}.
\end{equation}
Hence, from (\ref{3.4}) for $n=0,1,\ldots $
\begin{equation} \label{3.6}
\int_{n/\sqrt{s}}^{(n+1)/\sqrt{s}}\,e^{-\gamma_sx}\,f(x)\,dx\geq\frac{1}{\sqrt{s}}\,e^{\gamma_s/\sqrt{s}}\,
e^{-\gamma_s\frac{n+1}{\sqrt{s}}}\,f\Bigl(\frac{n+1}{\sqrt{s}}\Bigr),
\end{equation}
and from (\ref{3.5}) for $n=0,1,\ldots $
\begin{equation} \label{3.7}
\int_{(n+1)/\sqrt{s}}^{(n+2)/\sqrt{s}}\,e^{-\gamma_sx}\,f(x)\,dx\leq\frac{1}{\sqrt{s}}\,e^{-\gamma_s/\sqrt{s}}\,
e^{-\gamma_s\frac{n+1}{\sqrt{s}}}\,f\Bigl(\frac{n+1}{\sqrt{s}}\Bigr).
\end{equation}
From (\ref{3.6}) and (\ref{3.7}) the two inequalities in (\ref{3.2}) readily follow.
\end{proof}

\refProposition{prop3.2} shows that $F_s(\rho)<\infty$ if and only if $\mathcal{L}(\gamma_s)<\infty$, where it is used that $f$ is non-negative and bounded. By the definition of $\gamma_{\min}$ and the assumption in (\ref{2.15}) it follows that $F_s(\rho)<\infty$ if and only if $\gamma_s>\gamma_{\min}$. Then from (\ref{3.3}) the equivalence in \refProposition{prop2.1}(i) follows.

\subsection{Proof of \refProposition{prop2.1}(ii)}
Let $s=1,2,\ldots$ and consider the case $-\gamma_{\min}=\lim_{x\to\infty}a(x)=: L < \infty$. Then, by monotonicity of $a$,
\begin{equation} \label{640}
q_s(n)=\prod_{k=0}^n\frac{1}{1+\frac{1}{\sqrt{s}}a(\frac{k+1/2}{\sqrt{s}})}\geq \Big(\frac{1}{1+\frac{1}{\sqrt{s}}L}\Big)^{n+1},
\end{equation}
and so $F_s(\rho)=\infty$ when $\rho\geq 1+\frac{1}{\sqrt{s}}L$. Next, take $0\leq \rho<1+\frac{1}{\sqrt{s}}L$. From $a=a(\frac{k+1/2}{\sqrt{s}})\leq L$, $\rho< 1+\frac{1}{\sqrt{s}}L$,
\eqan{ \label{641}
\frac{\rho}{ 1+\frac{1}{\sqrt{s}}a}&=1-\frac{ 1+\frac{1}{\sqrt{s}}L-\rho}{ 1+\frac{1}{\sqrt{s}}a}+\frac{1}{\sqrt{s}}\frac{L-a}{ 1+\frac{1}{\sqrt{s}}a}\nonumber\\
&<1-\Big(1-\frac{\rho}{ 1+\frac{1}{\sqrt{s}}L}\Big)+\frac{1}{\sqrt{s}}
(L-a).
}
Hence, we can find a $K=1,2\ldots$ such that
\begin{equation} \label{642}
\frac{\rho}{1+\frac{1}{\sqrt{s}}a(\frac{k+1/2}{\sqrt{s}})}\leq 1-\frac{1}{2}\Big(1-\frac{\rho}{ 1+\frac{1}{\sqrt{s}}L}\Big), \quad k>K.
\end{equation}
Therefore,
\begin{align} \label{643}
F_s(\rho)&\leq \sum_{n=0}^K \rho^{n+1}+\rho^{K+1}\sum_{n=K+1}^\infty \rho^{n-K} \prod_{k=K+1}^n \frac{1}{1+\frac{1}{\sqrt{s}}a(\frac{k+1/2}{\sqrt{s}})}\nonumber\\
&<\sum_{n=0}^K \rho^{n+1}+\rho^{K+1}\sum_{n=K+1}^\infty \Big(1-\frac{1}{2}\Big(1-\frac{\rho}{ 1+\frac{1}{\sqrt{s}}L}\Big)\Big)^{n-K}<\infty.
\end{align}
This proves the result for the case $L<\infty$. The proof for the case $L=\infty$ is similar.

\section{Proof of \refProposition{prop:Weak_convergence_to_a_diffusion_process}}\label{proof:dp}
We will use Stone's theorem \cite{Iglehart73}, for which we need to verify that (a) the state space of the normalized process converges to a limit that is dense in $\realNumbers$ and (b) the infinitesimal mean and variance of $\process{\vC{X}{s}(t)}{t \geq 0}$ converge uniformly to $m(x)$ and $\sigma^2(x)$, respectively.

Condition (a) is readily verified. The state space of $\process{\vC{X}{s}(t)}{t \geq 0}$ is given by $\vC{\stateSpace}{s} = \{ (k-s)/\sqrt{s} | k \in \naturalNumbersZero\}$, and we see that for every $x \in \realNumbers$ and every $\epsilon > 0$, there exists an $s > 0$ such that $\min_{ y \in \vC{\stateSpace}{s} } | x - y | < \epsilon$.

To verify condition (b), we recall that for any birth--death process $\process{Y(t)}{t \geq 0}$ with birth--death parameters $\seq{\lambda}{k}$ and $\seq{\mu}{k}$ and associated states
\begin{equation}
\seq{\sB}{0} < \seq{\sB}{1} < \seq{\sB}{2} < \ldots,
\end{equation}
the infinitesimal mean and variance are defined as
\begin{equation}
m(\sB) = \seq{\lambda}{ e(\sB) } ( \seq{\sB}{ e(\sB) + 1 } - \seq{\sB}{e(\sB)} ) - \seq{\mu}{e(\sB)} ( \seq{\sB}{e(\sB)} - \seq{\sB}{e(\sB)-1} )
\end{equation}
and
\begin{equation}
\sigma^2(\sB) = \seq{\lambda}{ e(\sB) } ( \seq{\sB}{ e(\sB) + 1 } - \seq{\sB}{ e(\sB) } )^2 + \seq{\mu}{ e(\sB) } ( \seq{\sB}{ e(\sB) } - \seq{\sB}{ e(\sB) - 1 } )^2,
\end{equation}
respectively. Here,
\begin{equation}
e(\sB) = \arg \sup_{ k \in \naturalNumbersZero } \{ \seq{\sB}{k} | \seq{\sB}{k} \leq \sB \}
\end{equation}
is the label of the state closest to, but never above, $y \in [ \seq{y}{0}, \seq{y}{\infty} )$.

For each birth--death process $\process{\vC{X}{s}(t)}{t \geq 0}$, we have that
\begin{gather}
\seq{\vC{\lambda}{s}}{k} = \vC{\lambda}{s} \indicator{ k < s } + \vC{\lambda}{s} \vC{p}{s}(k-s) \indicator{k \geq s}, \quad k \in \naturalNumbersZero, \\
\seq{\vC{\mu}{s}}{k} = \min \{ k, s \}, \quad k \in \naturalNumbersZero, \\
\vC{\seq{\sA}{ \vC{e}{s}(x) + 1 }}{s} - \vC{\seq{\sA}{ \vC{e}{s}(x) }}{s} = 1 / \sqrt{s}, \quad x \in [ - \sqrt{s}, \infty ),
\end{gather}
and
\begin{equation}
\vC{e}{s}(x) = \lfloor s + x \sqrt{s} \rfloor, \quad x \in [ - \sqrt{s}, \infty ).
\end{equation}
Because $\gamma$ is fixed, \refEquation{2.8} prescribes that we are scaling the arrival rate as $\vC{\lambda}{s} = s - \gamma \sqrt{s}$. This yields
\begin{equation}
\vC{m}{s}(x)
=
\begin{cases}
- \gamma + \frac{ s - \lfloor s + \sqrt{s} x \rfloor }{ \sqrt{s} }, & x < 0, \\
- \gamma \vC{p}{s}( \lfloor s + \sqrt{s} x \rfloor - s ) + ( \vC{p}{s}( \lfloor s + \sqrt{s} x \rfloor - s ) - 1 ) \sqrt{s}, & x \geq 0, \\
\end{cases} \label{eqn:Infinitesimal_drift_of_process_Xs}
\end{equation}
and
\begin{equation}
\vC{\sigma}{s}^2(x)
=
\begin{cases}
\frac{ s + \lfloor s + \sqrt{s} x \rfloor ) }{ s } - \frac{ \gamma }{ \sqrt{s} }, & x < 0, \\
\vC{p}{s}( \lfloor s + \sqrt{s} x \rfloor - s ) + 1 - \frac{ \gamma \vC{p}{s}( \lfloor s + \sqrt{s} x \rfloor - s ) }{ \sqrt{s} }, & x \geq 0. \\
\end{cases} \label{eqn:Infinitesimal_variance_of_process_Xs}
\end{equation}

By first Taylor expanding \refEquation{1.2},
\begin{equation}
p_s(k)
= \frac{1}{1+\frac{1}{\sqrt{s}}a(\frac{k+1}{\sqrt{s}})}
= 1 - \frac{1}{\sqrt{s}} a \Bigl( \frac{k+1}{\sqrt{s}} \Bigr) + \bigO{ s^{-1} }, \quad k \in \naturalNumbersZero, \label{eqn:Taylor_expansion_of_local_probability_control}
\end{equation}
and then substituting \refEquation{eqn:Taylor_expansion_of_local_probability_control} into \refEquation{eqn:Infinitesimal_drift_of_process_Xs} and \refEquation{eqn:Infinitesimal_variance_of_process_Xs}, we conclude that
\begin{equation}
\vC{m}{s}(x)
=
\begin{cases}
- \gamma - \frac{ \lfloor s + \sqrt{s} x \rfloor - s }{ \sqrt{s} }, & x < 0, \\
- \gamma - a \bigl( \frac{ \lfloor s + \sqrt{s} x \rfloor - s + 1}{\sqrt{s}} \bigr) + \bigO{ s^{ -\frac{1}{2} } }, & x \geq 0, \\
\end{cases}
\end{equation}
and $\vC{\sigma}{s}^2(x) = 2 + \bigO{ s^{-\frac{1}{2}} }$ for all $x$.
Because $( \lfloor s + \sqrt{s} x \rfloor - s + 1 ) / \sqrt{s} \rightarrow x$ as $s \rightarrow \infty$ and $a$ is continuous and bounded on every compact subinterval $I$ of $\realNumbers$, we have that for every compact subinterval $I$ of $\realNumbers$, $\lim_{s \rightarrow \infty} \vC{m}{s}(x) = m(x)$
and $\lim_{s \rightarrow \infty} \vC{\sigma}{s}^2(x) = \sigma^2(x) = 2$ uniformly for $x \in I$. This concludes the proof.

\section{\glsentrytext{EM} summation} \label{appB}

Let $m=1,2,\ldots \,$, $N=1,2,\ldots \,$, and let $h\in C^{2m}([0,N+1])$. Then
\begin{align} \label{B.1}
&\sum_{n=0}^N h( n+ \nicefrac{1}{2} ) - \int_0^{N+1} h(x) dx \nonumber \\
&= \sum_{k=1}^m \frac{B_{2k}( \nicefrac{1}{2} )}{(2k)!} (h^{(2k-1)}(N{+}1){-}h^{(2k-1)}(0))
- \int_0^{N+1} \frac{\tilde{B}_{2m}(x - \nicefrac{1}{2} )}{(2m)!} h^{(2m)}(x) dx \nonumber \\
&= \sum_{k=1}^{m-1} \frac{B_{2k}( \nicefrac{1}{2} )}{(2k)!} (h^{(2k-1)}(N{+}1){-}h^{(2k-1)}(0))
- \int_0^{N+1} \frac{\tilde{B}_{2m}(x - \nicefrac{1}{2} )-B_{2m}( \nicefrac{1}{2} )}{(2m)!} h^{(2m)}(x) dx.
\end{align}
Here
\begin{equation} \label{B.2}
B_{2k}( \nicefrac{1}{2} ) = {-}(1-2^{-2k+1}) B_{2k}, \quad k = 1, 2, \ldots ,
\end{equation}
with $B_{2k}$ the Bernoulli numbers of positive, even order, see \cite[\S 24.2~(i)]{OLBC10} and $\tilde{B}_{2m}(x)=B_{2m}(x-\lfloor x\rfloor)$ with $B_{2m}(x)$ the Bernoulli polynomial of degree $2m$. Moreover, the two integrals in (\ref{B.1}) involving $\tilde{B}_{2m}$ can be estimated as
\begin{equation} \label{B.3}
\Bigl| \int_0^{N+1} \frac{\tilde{B}_{2m}(x- \nicefrac{1}{2} )}{(2m)!} h^{(2m)}(x) dx \Bigr|
\leq \frac{|B_{2m}|}{(2m)!} \int_0^{N+1} |h^{(2m)}(x)| dx
\end{equation}
and
\begin{equation} \label{B.4}
\Bigl| \int_0^{N+1} \frac{\tilde{B}_{2m}(x- \nicefrac{1}{2} )-B_{2m}( \nicefrac{1}{2} )}{(2m)!} h^{(2m)}(x) dx \Bigr|
\leq 2(1-2^{-2m}) \frac{|B_{2m}|}{(2m)!} \int_0^{N+1} |h^{(2m)}(x)| dx,
\end{equation}
respectively. These formulas follow from \cite[Theorem~1.3]{ref7} or \cite[Theorem~2.1]{ref8} (the latter reference containing a proof) for EM summation of $h$ sampled at points $n+\nu$, $n=0,1,\ldots, N$. For the special case that $\nu = \nicefrac{1}{2}$, a simplification occurs due to $B_j( \nicefrac{1}{2} )=0$ for $j=1,3,\ldots \,$. The bounds in (\ref{B.3}) and (\ref{B.4}) follow from \cite[\S 24.12~(i) and \S 24.4.34]{OLBC10}, with a special consideration for $B_2(x)=x^2-x+1/6$. We have $B_2=1/6$, $B_4=-1/30$, $B_6=1/42, \ldots \,$. When $m=1$, the series over $k$ in the second form in (\ref{B.1}) is absent.

The formula (\ref{B.1}) is sometimes called the second \gls{EM} summation formula, see \cite[(5.8.18--19) on p.~154]{ref9}. It distinguishes itself from the first \gls{EM} summation formula, as appears in \cite[ \S 2.10~(i)]{OLBC10}, in that
{\rm{(i)}} half-integer, rather than integer samples of $h$ are used at the left-hand side,
{\rm{(ii)}} absence of a term $\frac12\,(h(N+\nicefrac{1}{2})+h(\nicefrac{1}{2}))$ at the right-hand side, and
{\rm{(iii)}} smaller coefficients $B_{2k}(\nicefrac{1}{2})$ in the series over $k$ at the right-hand side.
Hence, also see the comment in \cite[(5.8.18--19)]{ref9}, the second EM formula is somewhat simpler in form and slightly more accurate when the remainder terms $R$ are dropped than the first EM formula.

In the main text, this formula is used for $h(x)=g( (x+ \nicefrac{1}{2} ) / \sqrt{s} )$, with $g\in C^{2m}([0,\infty))$ and $m=1$ and 2 while assuming that $\int_0^{\infty} |g^{(2m)}(x)| dx < \infty$ and that $g^{(2k-1)}(x)\pr0$, $x\pr\infty$, for $k=1$ and $k=1,2$, respectively. Subsequently, the formula is used for functions $g(x)$ of the form $\exp({-}\gamma_s x) f(x)$, $x\geq0$.

\section{Proof of \refTheorem{thm4.10}} \label{proofthm4.10}

We can write
\begin{equation} \label{4.37}
D_s(\rho) = \frac{1+F_s(\rho)}{B_s^{-1}(\rho) + F_s(\rho)} = H_1( F_s(\rho), B_s(\rho) )
\end{equation}
and
\begin{equation} \label{4.38}
D_s^R(\rho) = \frac{1+(1-\rho^{-1}) F_s(\rho)} {B_s^{-1}(\rho)+F_s(\rho)} = H_{1-\rho^{-1}}( F_s(\rho), B_s(\rho) ),
\end{equation}
where
\begin{equation} \label{4.39}
H_a(x,y) = \frac{1+ax}{y^{-1}+x}.
\end{equation}
Error propagation in $D_s$ and $D_s^R$ when both $F_s(\rho)$ and $B_s(\rho)$ are approximated can be assessed using the following result.

\begin{prop} \label{prop4.8}
For $a\in\dR$ and $x\geq0$, $x+\Delta x\geq0$ and $0\leq y\leq1$, $0\leq y+\Delta y\leq1$, it holds that
\begin{equation} \label{4.40}
|H_a(x+\Delta x,y+\Delta y)-H_a(x,y)| \leq |y(a-y)| |\Delta x|+|1+ax| |\Delta y|.
\end{equation}
\end{prop}
\begin{proof}
We have
\begin{equation} \label{4.41}
\frac{\partial H_a}{\partial x} = \frac{y(a-y)}{(1+xy)^2}, \quad \frac{\partial H_a}{\partial y} = \frac{1+ax}{(1+xy)^2}.
\end{equation}
Therefore
\begin{equation} \label{4.42}
\max_{x\geq0} \Bigl| \frac{\partial H_a}{\partial x} \Bigr|=|y(a-y)|, \quad 0 \leq y \leq 1,
\end{equation}
and
\begin{equation} \label{4.43}
\max_{0\leq y\leq1} \Bigl| \frac{\partial H_a}{\partial y} \Bigr| = |1+ax|, \quad x \geq 0,
\end{equation}
and the result follows.
\end{proof}

We insert the approximations (\ref{4.33}) and (\ref{4.34}) into (\ref{4.37}), and we get
\begin{equation} \label{4.48}
D_s(\rho) = \Bigl(\frac{1}{\sqrt{s}} g+\frac1s h\Bigr) \frac{1+\sqrt{s} {\cal L}+{\cal M}} {1+ (\frac{1}{\sqrt{s}} g + \frac1s h ) (\sqrt{s}{\cal L}+{\cal M})} + O\Bigl( \frac1s \Bigr).
\end{equation}
The approximation error $O(1/s)$ here is obtained from using (\ref{4.40}) with $a=1$ and with
\begin{equation} \label{4.49}
x=\sqrt{s}\,{\cal L}+{\cal M}=O(\sqrt{s}),\quad \Delta x=O\Bigl(\frac{1}{\sqrt{s}}\Bigr),
\end{equation}
\begin{equation} \label{4.50}
y=\frac{1}{\sqrt{s}}\,g+\frac1s\,h=O\Bigl(\frac{1}{\sqrt{s}}\Bigr),\quad \Delta y=O\Bigl(\frac{1}{s\sqrt{s}}\Bigr),
\end{equation}
where
the $O$'s in (\ref{4.49}--\ref{4.50}) hold uniformly in any compact set of $\gamma$'s contained in $(\gamma_{\min},\infty)$. Consequently, the $O(1/s)$ in (\ref{4.48}) holds uniformly in any compact set of $\gamma$'s contained in $(\gamma_{\min},\infty)$. Expanding the expression at the right-hand side of (\ref{4.48}), retaining only the terms $O(1)$ and $O(1/\sqrt{s})$, then yields (\ref{4.44}--{\ref{4.45}).

In a similar fashion (\ref{4.46}--\ref{4.47}) is shown, although the computations are rather involved. We must be a bit careful with $T_2^R$ because of the denominator $1-\gamma{\cal L}$ that appears in (\ref{4.47}). Recall that in the case that $f(x)=1$, $x\geq0$, we have that $1-\gamma{\cal L}=0=\gamma{\cal L}+{\cal M}$, and so $T_1^R=T_2^R=0$. In the case that $f(x_0)<1$ for some $x_0\geq0$, it is easy to show from $f(0)=1$, non-negativity and decreasingness of $f(x)$ that for any compact $C\subset(\gamma_{\min},\infty)$
\begin{equation} \label{4.51}
\max_{\gamma\in C}\,\gamma{\cal L}(\gamma)<1.
\end{equation}
This yields uniform validity of the $O(1/s\sqrt{s})$ in (\ref{4.46}) when $\gamma$ is restricted to a compact subset of $(\gamma_{\min},\infty)$.

\section{Proof of \refProposition{propallA}}\label{proofpropallA}

We assume that $a(x)$ is non-negative and non-decreasing.
Define
\begin{equation} \label{644}
a^{\leftarrow}(y) = \sup{ \{ x \geq 0 \, | \, a(x) \leq y \} }
\end{equation}
for $y\geq a(0)$. This generalized inverse function is continuous from the right at all $y$ such that $a^{\leftarrow}(y)$ is finite. Furthermore note that $a(x)\leq y$ when $x<a^{\leftarrow}(y)$.

\begin{lem} \label{lemmm}
Let $s = 1, 2, \ldots$, and denote for $n = 1, 2, \ldots$
\begin{equation} \label{645}
S_s(n)=\sum_{k=0}^n \ln \Big(1+\frac{1}{\sqrt{s}}a\Big(\frac{k+1}{\sqrt{s}}\Big)\Big).
\end{equation}
Also, let $\delta\in(0,1/2)$. Then
\begin{equation} \label{646}
0\leq S_s(n)-\sqrt{s}\int_{0}^{\frac{n+1}{\sqrt{s}}} \ln \Big(1+\frac{1}{\sqrt{s}}a(x)\Big)dx
\leq \frac12 s^{\delta- \frac{1}{2} }.
\end{equation}
when $n+1 < \sqrt{s} a^{\leftarrow}( s^\delta / 2 )$. Furthermore, define
\begin{equation} \label{647}
A(x)=\int_0^x a^2(u)du, \quad x\geq 0.
\end{equation}
Then, except in the trivial case $a \equiv 0$, $A(x)$ is continuous and strictly increasing from $0$ at $x_0:=\sup{ \{ x\geq 0 \, | \, a(x)=0 \} }$ to $\infty$ at $x=\infty$. Furthermore,
\begin{align} \label{648}
0&\leq \int_{0}^{\frac{n+1}{\sqrt{s}}} a(x)dx-\sqrt{s}\int_{0}^{\frac{n+1}{\sqrt{s}}} \ln \Big(1+\frac{1}{\sqrt{s}}a(x)\Big)dx\leq \frac12 s^{\delta- \frac{1}{2} }
\end{align}
when $n+1<\sqrt{s}A^{\leftarrow}(s^{ \frac{1}{2} -\delta})$.
\end{lem}
\begin{proof}
Since $a(x)$ is non-decreasing in $x\geq 0$, we have that
 $S_s(n)/\sqrt{s}$ is an upper Riemann sum for
\begin{equation} \label{F6}
\int_{0}^{\frac{n+1}{\sqrt{s}}} \ln \Big(1+\frac{1}{\sqrt{s}}a(x)\Big)dx,
\end{equation}
while
\begin{equation} \label{F7}
\frac{1}{\sqrt{s}} S_s(n-1) = \frac{1}{\sqrt{s}} S_s(n) - \frac{1}{\sqrt{s}} \ln{ \Big(1+\frac{1}{\sqrt{s}}a\Big(\frac{n+1}{\sqrt{s}}\Big)\Big) }
\end{equation}
is a lower Riemann sum for
\begin{equation} \label{F8}
\int_{ \frac{1}{\sqrt{s}} }^{\frac{n+1}{\sqrt{s}}} \ln \Big(1+\frac{1}{\sqrt{s}}a(x)\Big)dx.
\end{equation}
It follows that
\begin{align} \label{F9}
&\sqrt{s}\int_{0}^{\frac{n+1}{\sqrt{s}}} \ln \Big(1+\frac{1}{\sqrt{s}}a(x)\Big)dx
\leq S_s(n)\nonumber\\
&\leq \sqrt{s}\int_{ \frac{1}{\sqrt{s}} }^{\frac{n+1}{\sqrt{s}}} \ln \Big(1+\frac{1}{\sqrt{s}}a(x)\Big)dx +\ln \Big(1+\frac{1}{\sqrt{s}}a\Big(\frac{n+1}{\sqrt{s}}\Big)\Big)\nonumber\\
&\leq \sqrt{s}\int_{0}^{\frac{n+1}{\sqrt{s}}} \ln \Big(1+\frac{1}{\sqrt{s}}a(x)\Big)dx +\ln \Big(1+\frac{1}{\sqrt{s}}a\Big(\frac{n+1}{\sqrt{s}}\Big)\Big),
\end{align}
where in the last inequality $a(x)\geq 0$ has been used. Now
\begin{equation} \label{F10}
\ln \Big(1+\frac{1}{\sqrt{s}}a\Big(\frac{n+1}{\sqrt{s}}\Big)\Big)\leq \frac{1}{\sqrt{s}}a\Big(\frac{n+1}{\sqrt{s}}\Big)\leq \frac12 s^{\delta- \frac{1}{2} }
\end{equation}
when $n+1<\sqrt{s}a^{\leftarrow}( s^{\delta} / 2 )$. This yields \refEquation{646}.

As for \refEquation{648}, we note that
\begin{equation} \label{651}
a(x)-\frac{1}{2\sqrt{s}}a^2(x)\leq \sqrt{s}\ln \Big(1+\frac{1}{\sqrt{s}}a(x)\Big)\leq a(x).
\end{equation}
Hence,
\begin{equation} \label{652}
0\leq \int_{0}^{\frac{n+1}{\sqrt{s}}} a(x)dx-\sqrt{s}\int_{0}^{\frac{n+1}{\sqrt{s}}} \ln \Big(1+\frac{1}{\sqrt{s}}a(x)\Big)dx
\leq \frac{1}{2\sqrt{s}} \int_{0}^{\frac{n+1}{\sqrt{s}}}a^2(x)dx\leq \frac12 s^{-\delta}
\end{equation}
when $n+1<\sqrt{s}A^{\leftarrow}(s^{ \frac{1}{2} - \delta})$.
\end{proof}

The proof of \refProposition{propallA} follows now from \refLemma{lemmm} with $\delta=1/4$ and taking $\psi(s) = \min \{ a^{\leftarrow}( s^{\delta} / 2 ), A^{\leftarrow}(s^{ \frac{1}{2} - \delta }) \}$.

\end{document}